\documentclass[10pt,reqno]{amsart}

\usepackage[a4paper, margin=1.5in]{geometry}

\usepackage {color}
\usepackage{amsmath,amsthm,amssymb}
\usepackage{epsfig}
\usepackage{graphicx}
\usepackage{amssymb}
\usepackage{latexsym}
\usepackage{pdfpages}

\usepackage{pgf}

\usepackage{ulem} 

\usepackage{ifpdf}
\newcommand{\res}{\!\!\mathop{\hbox{
                                \vrule height 7pt width .5pt depth 0pt
                                \vrule height .5pt width 6pt depth 0pt}}
                                \nolimits}

\def\z{{\bf z}}

\ifpdf 
  \usepackage[hidelinks]{hyperref}
\else 
\fi 

 \usepackage[usenames,dvipsnames]{pstricks}
 \usepackage{epsfig}
 \usepackage{pst-grad} 
 \usepackage{pst-plot} 

\usepackage{color}

\newtheorem{theorem}{Theorem}[section]
\newtheorem{lemma}[theorem]{Lemma}
\newtheorem{definition}[theorem]{Definition}
\newtheorem{proposition}[theorem]{Proposition}
\newtheorem{corollary}[theorem]{Corollary}
\newtheorem{remark}[theorem]{Remark}
\newtheorem{example}[theorem]{Example}

\newtheorem*{theorem*}{\it Theorem}

\def\vint_#1{\mathchoice%
          {\mathop{\kern 0.2em\vrule width 0.6em height 0.69678ex depth -0.58065ex
                  \kern -0.8em \intop}\nolimits_{\kern -0.4em#1}}%
          {\mathop{\kern 0.1em\vrule width 0.5em height 0.69678ex depth -0.60387ex
                  \kern -0.6em \intop}\nolimits_{#1}}%
          {\mathop{\kern 0.1em\vrule width 0.5em height 0.69678ex
              depth -0.60387ex
                  \kern -0.6em \intop}\nolimits_{#1}}%
          {\mathop{\kern 0.1em\vrule width 0.5em height 0.69678ex depth -0.60387ex
                  \kern -0.6em \intop}\nolimits_{#1}}}
\def\vintslides_#1{\mathchoice%
          {\mathop{\kern 0.1em\vrule width 0.5em height 0.697ex depth -0.581ex
                  \kern -0.6em \intop}\nolimits_{\kern -0.4em#1}}%
          {\mathop{\kern 0.1em\vrule width 0.3em height 0.697ex depth -0.604ex
                  \kern -0.4em \intop}\nolimits_{#1}}%
          {\mathop{\kern 0.1em\vrule width 0.3em height 0.697ex depth -0.604ex
                  \kern -0.4em \intop}\nolimits_{#1}}%
          {\mathop{\kern 0.1em\vrule width 0.3em height 0.697ex depth -0.604ex
                  \kern -0.4em \intop}\nolimits_{#1}}}

\def\R{\mathbb R}
\def\N{\mathbb N}
\def\g{\hbox{\bf g}}

\numberwithin{equation}{section}

\def\NN{{\mathbb{N}}}

\def\1{\raisebox{2pt}{\rm{$\chi$}}}

\def\g{{\bf g}}

\definecolor{violet(ryb)}{rgb}{0.53, 0.0, 0.69}
\usepackage{mathtools}
\mathtoolsset{showonlyrefs}

\begin{document}  

\title[Torsional rigidity in random walk spaces]{\bf Torsional rigidity in random walk spaces}

\author[J. M. Maz\'{o}n and J. Toledo]{J. M. Maz\'{o}n    and J. Toledo     }

\address{Departamento de An\'{a}lisis Matem\'atico,
Universitat de Val\`encia, Valencia, Spain.\newline
  \indent {\tt mazon@uv.es, toledojj@uv.es }
  }


\date{}

\begin{abstract} In this paper we study the (nonlocal) torsional rigidity in the ambient space of random walk spaces.   We get  the relation of the (nonlocal) torsional rigidity of a set $\Omega$ with the spectral $m$-heat content of $\Omega$,  what gives rise to  a complete description of  the nonlocal torsional rigidity of $\Omega$ by using uniquely  probability terms involving the set $\Omega$; and  recover the  first  eigenvalue of the nonlocal Laplacian with homogeneous Dirichlet boundary conditions  by a limit formula using these probability term. For the random walk in $\R^N$ associated with a non singular kernel,   we get a nonlocal version  of the Saint-Venant inequality, and, under rescaling we recover the classical Saint-Venant inequality. We study the nonlocal $p$-torsional rigidity and its relation with the nonlocal Cheeger constants.   We also get a nonlocal version of the P\'{o}lya-Makai-type inequalities.  We relate the torsional rigidity  given here for weighted graphs with the torsional rigidity on metric graphs.

\end{abstract}

\maketitle

{ \renewcommand\contentsname{Contents}
\setcounter{tocdepth}{3}
}

\addtolength{\parskip}{0.2cm}

\section{Introduction}

In this paper we study the (nonlocal) torsional rigidity in the ambient space of random walk spaces.  Important examples of these spaces are
locally finite weighted graphs, finite Markov chains and nonlocal operators on domains in  $\R^N$ where the jumps are driven by a non-negative  integrable and radially symmetric kernel (see~\cite{MST0} and \cite{MSTBook}).

In the classical local setting, the torsional rigidity of a Lebesgue subset of $\mathbb{R}^N$ has been, and  is nowadays, a source of interesting problems.
Let us consider an isotropic elastic cylindrical beam in $\mathbb{R}^3$ with cross-section, perpendicular to the $z$-axis, is an an open
bounded domain $D \subset \R^2$.   The  {\it torsion rigidity  problem} (see e.g. \cite {Sokolnikoff}) is to find the shape of the cross section $D$ which provides the greatest {\it torsional rigidity}, under an area constraint, when a  torque is applied around the $z$-axis . It was  conjectured by   A. Saint-Venant    in 1856 that the simply connected cross-section with maximal torsional rigidity is the circle and it was proved by G. P\'{o}lya in 1948. The distribution of stress generated in the beam due to  the applied
torque is determined by the {\it stress function}  $u_D$,  the unique positive weak solution of the Dirichlet problem
\begin{equation}\label{DPTorsion}
\left\{ \begin{array}{ll} - \Delta u_D = 1  \quad \hbox{in}  \ D \\[10pt]
u_D= 0  \quad \hbox{on} \ \partial D. \end{array} \right.
\end{equation}
Notice that the function $u_D$ is also the unique minimizer of the {\it torsional energy}
 $$E(D) = \min_{v\in W^{1,2}_0(D)}  \frac12 \int_D |\nabla v|^2 dx - \int_D v dx.$$
The total resultant torque due to this stress function is called {\it torsional rigidity} and is
expressed as
$$T(D):= \int_\Omega u_D(x) dx$$
or equivalently (see~\cite{P1} or \cite{Bandle})
\begin{equation}\label{var1}T(D) = \displaystyle\max_{v\in W^{1,2}_0(D)\setminus\{0\}}\frac{\displaystyle \left(\int_D v dx \right)^2}{\displaystyle\int_D |\nabla v|^2 dx}.
\end{equation}

Throughout this paper, we adopt the following notation. If $D$ is open in $\R^N$ with $0 < \vert D \vert < \infty$ then $D^*$ is the ball in $\R^N$ centered at the origin with $\vert D^* \vert = \vert D \vert$. Furthermore $B_R$ is a ball with radius $R$. We put $\omega_N = \vert B_1 \vert$.

The {\it Saint-Venant inequality} reads, for $D$ a bounded domain, as follows:
$$T(D) \leq T(D^*).$$
This inequality was established by G. P\'{o}lya \cite{P1} using symmetrization methods (see also E. Makai \cite{M}).

 On the other and,
the {\it Faber-Krahn inequality} establishes that
$$\lambda_1(D^*) \leq \lambda_1(D),$$
where $\lambda_1(D)$ is the lowest $\lambda$ for which the eigenvalue problem
\begin{equation}\label{Eigen}
\left\{ \begin{array}{ll} - \Delta u = \lambda u, \quad \hbox{in}  \ D \\[10pt]
u = 0, \quad \hbox{on} \ \partial D. \end{array} \right.
\end{equation}
admits a non trivial solution. The first proof of the Faber-Krahn inequality was given by P\'{o}lya and Szeg\"o in \cite{PS1}  based in spherically symmetric decreasing rearrangement.

Since $\lambda_1(D)$ is the minimizer  of the Rayleigh quotient
$$\lambda_1(D) = \min_{v\in W^{1,2}_0(D)\setminus\{0\}}\frac{\displaystyle \int_D |\nabla v|^2 dx}{\displaystyle\int_D v^2 dx },$$
is easy to see (see, for example, \cite{BBV}) that
\begin{equation}\label{lambdatorsionlocal1}\lambda_1(D)\le\frac{|D|}{T(D)}.
\end{equation}

Let $D$ be an open bounded domain $D \subset \R^N$. The {\it spectral heat content} of $D$ is given by $$\mathbb{Q}_D(t):= \int_D   v_D(x, t)  dx$$
where $v_D$ is the solution of Dirichlet problem
$$\left\{ \begin{array}{lll}   \displaystyle\frac{v_D(x, t)}{\partial t}(x, t)  = \Delta  v_D(x, t)(x,t), \quad &\hbox{if} \ (x,t) \in D \times [0, \infty), \\[10pt]   v_D(x, t) (x,t) =0, \quad &\hbox{if} \ (x,t) \in \partial D \times (0, \infty), \\[10pt]  v_D(x, t)(x,0) = \1_D(x), \quad &\hbox{if} \ x\in D.  \end{array}\right. $$
$\mathbb{Q}_D(t)$ represents the amount of heat contained in $D$ at time $t$ when $D$ has initial temperature $1$ and when the boundary of $D$ is keps at temperature $0$ for all $t >0$.

  The functions $u_D$ and $v_D$ have a probabilistic interpretation (see for instance \cite{BBC}). For this, let $(B(s), s \geq 0, \mathbb{P}_x, x \in\R^N)$ be a brownian motion associated to the Laplacian on $\R^N$,  and let $\tau$ be the first exit time from $D$:
$$\tau = \inf \{ s \geq 0 \ : \ B(s) \not\in D \}.$$
Then
\begin{equation}\label{exp1}u_D(x) = \mathbb{E}_x[\tau], \quad x \in D,
\end{equation}
where $\mathbb{E}_x$ denotes expectation with respect to $ \mathbb{P}_x$, and
\begin{equation}\label{exp2}v_D(x,t) = \mathbb{P}_x[\tau >t], \quad x \in D, \ t>0.\end{equation}
For $j \in \N$ the {\it sequence of exit-moments} of $D$ is defined as
$$EM_j(D):= \int_D \mathbb{E}_x[\tau^j] \, dx.$$
Notice that, by \eqref{exp1},
\begin{equation}\label{exp3}T(D) = EM_1(D).
\end{equation}
Using \eqref{exp2}, we can express moments of the exit time in term of $v_D$ as
\begin{equation}\label{exp4}
\mathbb{E}_x[\tau^j] = j \int_0^\infty t^{j-1} v_D(x, t) dt.
\end{equation}
Integrating in \eqref{exp4} and using Fubini's Theorem, we see that  the sequence of exit-moments can be expressed as moments of the heat content:
\begin{equation}\label{exp5}
EM_j(D) = j \int_0^\infty t^{j-1} \mathbb{Q}_D(t) dt.
\end{equation}
In particular, by \eqref{exp3}, we have
\begin{equation}\label{exp6}
T(D) = \int_0^\infty  \mathbb{Q}_D(t) dt.
\end{equation}

Our aim is  to study the torsional rigidity  in the general framework of the random walk spaces. We get the nonlocal versions of the previous local results \eqref{var1}, \eqref{lambdatorsionlocal1}, \eqref{exp3} and \eqref{exp6}.
 In particular we give the precise characterization of the nonlocal torsional rigidity of a set, and of the all nonlocal exit moments, by using uniquely  probability terms involving the set, see~\eqref{specheatcont01t02} and~\eqref{Nspecheatcont01t02mom}, and  recover the    first  eigenvalue of the nonlocal Laplacian with homogeneous Dirichlet boundary conditions, when exists, by a limit formula using such terms,  see~\eqref{pueq010}.
For the random walk in $\R^N$ associated with a non singular kernel,   we get a nonlocal version  of the Saint-Venant inequality, and, under rescaling we recover the classical Saint-Venant inequality.
We also get  the variational characterization of the nonlocal $p$-torsional rigidity. We  relate  the nonlocal $p$-torsional rigidity of a set with its $1$-Cheeger  and $p$-Cheeger constants in~\eqref{ChegeerT1}, and as a consequence we prove   that the nonlocal $1$-Cheeger constant of a set is the limit, as $p\to 1^+$,  of the inverse of its  nonlocal $p$-torsional rigidities, see~\eqref{ChegeerT2}. See also~\eqref{ChegeerT2dd01} for another limit attaining the nonlocal $1$-Cheeger constant by means of nonlocal Poincar\'{e} constants.  We also obtain  a nonlocal version of P\'{o}lya-Makai-type inequalities.
To the best of our knowledge  most of the results we get are new even for the   particular cases of locally finite weighted graphs and nonlocal problems in domains of $\mathbb{R}^N$.
 Finally we relate the torsional rigidity  given here for graphs with the torsional rigidity on metric graphs stated in~\cite{MP}.

\section{Preliminaries}

 \subsection{Random walk spaces}\label{RWS1} We recall some concepts and results about random walk spaces  given in \cite{MST0},  \cite{MST2} and \cite{MSTBook}.

 Let $(X,\mathcal{B})$ be a measurable space such that the $\sigma$-field $\mathcal{B}$ is countably generated.
A random walk $m$
on $(X,\mathcal{B})$ is a family of probability measures $(m_x)_{x\in X}$
on $\mathcal{B}$ such that $x\mapsto m_x(B)$ is a measurable function on $X$ for each fixed $B\in\mathcal{B}$.

The notation and terminology chosen in this definition comes from Ollivier's paper \cite{O}. As noted in that paper, geometers may think of $m_x$ as a replacement for the notion of balls around $x$, while in probabilistic terms we can rather think of these probability measures as defining a Markov chain whose transition probability from $x$ to $y$ in $n$ steps is
\begin{equation}
\displaystyle
dm_x^{*n}(y):= \int_{z \in X}  dm_z(y)dm_x^{*(n-1)}(z), \ \ n\ge 1
\end{equation}
and $m_x^{*0} = \delta_x$, the dirac measure at $x$.

\begin{definition}\label{convolutionofameasure}{\rm
If $m$ is a random walk on $(X,\mathcal{B})$ and $\mu$ is a $\sigma$-finite measure on $X$. The convolution of $\mu$ with $m$ on $X$ is the measure defined as follows:
$$\mu \ast m (A) := \int_X m_x(A)d\mu(x)\ \ \forall A\in\mathcal{B},$$
which is the image of $\mu$ by the random walk $m$.}
\end{definition}

\begin{definition}\label{def.invariant.measure} {\rm If $m$ is a random walk on $(X,\mathcal{B})$,
a $\sigma$-finite measure $\nu$ on $X$ is {\it invariant}
with respect to the random walk $m$ if
 $$\nu\ast m = \nu.$$

The measure $\nu$ is said to be {\it reversible} if moreover, the detailed balance condition $$dm_x(y)d\nu(x)  = dm_y(x)d\nu(y) $$ holds true.}
\end{definition}

\begin{definition}\label{DefMRWSf}{\rm
Let $(X,\mathcal{B})$ be a measurable space where the $\sigma$-field $\mathcal{B}$ is countably generated. Let $m$ be a random walk on $(X,\mathcal{B})$ and $\nu$ an invariant measure with respect to $m$. The measurable space together with $m$ and $\nu$ is then called a random walk space
and is denoted by $[X,\mathcal{B},m,\nu]$.}
\end{definition}

 If $(X,d)$ is a Polish metric space (separable completely metrizable topological space), $\mathcal{B}$ is its Borel $\sigma$-algebra and $\nu$ is a Radon measure (i.e. $\nu$ is inner regular
and locally finite), then we denote  $[X,\mathcal{B},m,\nu]$ as $[X,d,m,\nu]$,   and call it a metric random walk space.

\begin{definition}\label{def.m.connected.random.walk.space}{\rm
Let $[X,\mathcal{B},m,\nu]$ be a random walk space. We say that $[X,\mathcal{B},m,\nu]$ is $m$-connected
if, for every $D\in \mathcal{B}$ with $\nu(D)>0$ and $\nu$-a.e. $x\in X$,
$$\sum_{n=1}^{\infty}m_x^{\ast n}(D)>0.$$
}
\end{definition}

 \begin{definition}\label{def.m.interaction}{\rm
Let  $[X,\mathcal{B},m,\nu]$ be a random walk space and let $A$, $B\in\mathcal{B}$. We define the {\it $m$-interaction} between $A$ and $B$ as
\begin{equation}\label{m.interaction} L_m(A,B):= \int_A \int_B dm_x(y) d\nu(x)=\int_A m_x(B) d\nu(x).
 \end{equation}
 }
 \end{definition}

  The following result gives a characterization of $m$-connectedness in terms of the $m$-interaction between sets.

\begin{proposition}\label{connectedness.iff.Lm}(\cite[Proposition 2.11]{MST0}, \cite[Proposition 1.34]{MSTBook})
 Let $[X,\mathcal{B},m,\nu]$ be a random walk space.
The following statements are equivalent:
\item{ (i) } $[X,\mathcal{B},m,\nu]$ is $m$-connected.
\item {(ii)} If $ A,B\in\mathcal{B}$ satisfy $A\cup B=X$ and $L_m(A,B)= 0$, then either $\nu(A)=0$ or $\nu(B)=0$.
\item {(iii)} If $A\in \mathcal{B}$ is a $\nu$-invariant set then either $\nu(A)=0$ or $\nu(X\setminus A)=0$.
\end{proposition}

\begin{definition}\label{defomegaconnected}
Let $[X,\mathcal{B},m,\nu]$ be a reversible random walk space, and let $\Omega\in\mathcal{B}$ with $\nu(\Omega)>0$. We denote by  $\mathcal{B}_\Omega$ to the following $\sigma$-algebra
 $$\mathcal{B}_\Omega:=\{B\in\mathcal{B} \, : \, B\subset \Omega\}.$$

We say that $\Omega$ is {\it $m$-connected
(with respect to $\nu$)} if $L_m(A,B)>0$ for every pair of non-$\nu$-null   sets $A$, $B\in \mathcal{B}_\Omega$ such that $A\cup B=\Omega$.
\end{definition}

Let us see now some examples of random walk spaces.

 \begin{example}\label{example.nonlocalJ} \rm
Consider the metric measure space $(\R^N, d, \mathcal{L}^N)$, where $d$ is the Euclidean distance and $\mathcal{L}^N$ the Lebesgue measure on $\R^N$ (which we will also denote by $|.|$). For simplicity, we will write $dx$ instead of $d\mathcal{L}^N(x)$. Let  $J:\R^N\to[0,+\infty[$ be a measurable, nonnegative and radially symmetric
function  verifying  $\int_{\R^N}J(x)dx=1$. Let $m^J$ be the following random walk on $(\R^N,d)$:
$$m^J_x(A) :=  \int_A J(x - y) dy \quad \hbox{ for every $x\in \R^N$ and every Borel set } A \subset  \R^N  .$$
Then, applying Fubini's Theorem it is easy to see that the Lebesgue measure $\mathcal{L}^N$ is reversible with respect to $m^J$. Therefore, $[\R^N, d, m^J, \mathcal{L}^N]$ is a reversible metric random walk space.
\end{example}

\begin{example}\label{example.graphs}[Weighted discrete graphs] \rm Consider a locally finite  weighted discrete graph $$G = (V(G), E(G)),$$ where $V(G)$ is the vertex set, $E(G)$ is the edge set and each edge $(x,y) \in E(G)$ (we will write $x\sim y$ if $(x,y) \in E(G)$) has a positive weight $w_{xy} = w_{yx}$ assigned. Suppose further that $w_{xy} = 0$ if $(x,y) \not\in E(G)$.  Note that there may be loops in the graph, that is, we may have $(x,x)\in E(G)$ for some $x\in V(G)$ and, therefore, $w_{xx}>0$. Recall that a graph is locally finite if every vertex is only contained in a finite number of edges.

 A finite sequence $\{ x_k \}_{k=0}^n$  of vertices of the graph is called a {\it  path} if $x_k \sim x_{k+1}$ for all $k = 0, 1, ..., n-1$. The {\it length} of a path $\{ x_k \}_{k=0}^n$ is defined as the number $n$ of edges in the path. With this terminology, $G = (V(G), E(G))$ is said to be {\it connected} if, for any two vertices $x, y \in V$, there is a path connecting $x$ and $y$, that is, a path $\{ x_k \}_{k=0}^n$ such that $x_0 = x$ and $x_n = y$.  Finally, if $G = (V(G), E(G))$ is connected, the {\it graph distance} $d_G(x,y)$ between any two distinct vertices $x, y$ is defined as the minimum of the lengths of the paths connecting $x$ and $y$. Note that this metric is independent of the weights.

For $x \in V(G)$ we define the weight at $x$ as
$$d_x:= \sum_{y\sim x} w_{xy} = \sum_{y\in V(G)} w_{xy},$$
and the neighbourhood of $x$ as $N_G(x) := \{ y \in V(G) \, : \, x\sim y\}$. Note that, by definition of locally finite graph, the sets $N_G(x)$ are finite. When all the weights are $1$, $d_x$ coincides with the degree of the vertex $x$ in a graph, that is,  the number of edges containing $x$.

For each $x \in V(G)$  we define the following probability measure
\begin{equation}\label{discRW}m^G_x:=  \frac{1}{d_x}\sum_{y \sim x} w_{xy}\,\delta_y.\\ \\
\end{equation}
It is not difficult to see that the measure $\nu_G$ defined as
 $$\nu_G(A):= \sum_{x \in A} d_x,  \quad A \subset V(G),$$
is a reversible measure with respect to this random walk. Therefore, $[V(G),\mathcal{B},m^G,\nu_G]$ is a reversible random walk space being $\mathcal{B}$ is the $\sigma$-algebra of all subsets of $V(G)$. Moreover $[V(G),d_G,m^G,\nu_G]$ is a reversible metric random walk space.

\end{example}

\begin{example}\label{example.restriction.to.Omega} \rm Given a random walk  space $[X,\mathcal{B},m,\nu]$ and $\Omega \in \mathcal{B}$ with $\nu(\Omega) > 0$, let
$$m^{\Omega}_x(A):=\int_A d m_x(y)+\left(\int_{X\setminus \Omega}d m_x(y)\right)\delta_x(A) \quad \hbox{ for every } A\in\mathcal{B}_\Omega  \hbox{ and } x\in\Omega.
$$
Then, $m^{\Omega}$ is a random walk on $(\Omega,\mathcal{B}_\Omega)$ and it easy to see that $\nu \res \Omega$ is invariant with respect to $m^{\Omega}$. Therefore,  $[\Omega,\mathcal{B}_\Omega,m^{\Omega},\nu \res \Omega]$ is a random walk space. Moreover, if $\nu$ is reversible with respect to $m$ then $\nu \res \Omega$ is  reversible with respect to $m^{\Omega}$. Of course, if $\nu$ is a probability measure we may normalize $\nu \res \Omega$ to obtain the random walk space
$$\left[\Omega,\mathcal{B}_\Omega,m^{\Omega}, \frac{1}{\nu(\Omega)}\nu \res \Omega \right].$$
 Note that, if $[X,d,m,\nu]$ is a metric random walk space and $\Omega$ is closed, then $[\Omega,d,m^{\Omega},\nu \res \Omega]$ is also a metric random walk space, where we abuse notation and denote by $d$ the restriction of $d$ to $\Omega$.

In particular, in the context of Example \ref{example.nonlocalJ}, if $\Omega$ is a closed and bounded subset of $\R^N$, we obtain the metric random walk space $[\Omega, d, m^{J,\Omega},\mathcal{L}^N\res \Omega]$ where
$m^{J,\Omega} := (m^J)^{\Omega}$; that is,
$$m^{J,\Omega}_x(A):=\int_A J(x-y)dy+\left(\int_{\R^n\setminus \Omega}J(x-z)dz\right)d\delta_x$$ for every Borel set   $A \subset  \Omega$  and $x\in\Omega$.

\end{example}

\subsection{The nonlocal gradient, divergence and Laplace operators}\label{nonlocal.notions.1.section}

Let us introduce the nonlocal counterparts of some classical concepts.

\begin{definition}\label{nonlocalgraddiv}{\rm
Let $[X,\mathcal{B},m,\nu]$ be a random walk space. Given a function $f: X \rightarrow \R$ we define its {\it nonlocal gradient}
$\nabla f: X \times X \rightarrow \R$ as
$$\nabla f (x,y):= f(y) - f(x) \quad \forall \, x,y \in X.$$
Moreover, given $\z : X \times X \rightarrow \R$, its {\it $m$-divergence}
${\rm div}_m \z : X \rightarrow \R$ is defined as
 $$({\rm div}_m \z)(x):= \frac12 \int_{X} (\z(x,y) - \z(y,x)) dm_x(y).$$
 }
\end{definition}

We define the (nonlocal) Laplace operator as follows.
\begin{definition}\label{deflap1310}{\rm
Let $[X,\mathcal{B},m,\nu]$ be a random walk space, we define the {\it $m$-Laplace operator} (or {\it $m$-Laplacian}) from $L^1(X,\nu)$ into itself as $\Delta_m:= M_m - I$, i.e.,
$$\Delta_m f(x)= \int_X f(y) dm_x(y) - f(x) = \int_X (f(y) - f(x)) dm_x(y), \quad x\in X,$$
for $f\in L^1(X,\nu)$.
}
\end{definition}
 Note that
$$\Delta_m f (x) = {\rm div}_m (\nabla f)(x).$$

In the case of the random walk space associated with a locally finite weighted discrete graph $G=(V,E)$ (as defined in Example~\ref{example.graphs}), the $m^G$-Laplace operator coincides with the graph Laplacian (also called the normalized graph Laplacian) studied by many authors (see, for example, \cite{BJ}, \cite{BJL}, \cite{DK}, \cite{Elmoatazetal}, \cite{Hafiene}):
$$\Delta_{m^G} u(x):=\frac{1}{d_x}\sum_{y\sim x}w_{xy}(u(y)-u(x)), \quad u\in L^2(V,\nu_G), \ x\in V .$$

 In \cite{MST2} (see also \cite{MSTBook})  we define and proof the following facts.
 $$BV_m(X):= \left\{ f: X \rightarrow \R \ \hbox{measurable} \ : \ \int_{X \times X} \vert \nabla u(x,y) \vert \,  d(\nu\otimes m_x)(x,y) < \infty \right\},$$
and for $f \in BV_m(X)$ we define its {\it $m$-total variation} as
$$TV_m(f):= \frac12 \int_{X \times X} \vert \nabla u(x,y) \vert \,  d(\nu\otimes m_x)(x,y).$$
For a set $E \in \mathcal{B}$ such that $\1_E \in BV_m(X)$, we define its {\it $m$-perimeter} as
$$P_m(E):= TV_m(\1_E)   = L_m(E, X \setminus E).$$
 If $\nu(E)<+\infty$ then
 \begin{equation}\label{secondf021}\displaystyle P_m(E)=\nu(E) -\int_E\int_E dm_x(y) d\nu(x).
\end{equation}
The following {\it coarea formula} holds:
\begin{equation}\label{coarea}
TV_m(f) = \int_{-\infty}^{+\infty} P_m( \{ x\in X  :   f(x) >t \}) dt,\quad\hbox{for }f \in BV_m(X),
\end{equation}
  Furthermore we give the following nonlocal concept of mean curvature. Let $E \in \mathcal{B}$ with $\nu(E)>0$. For a point $x  \in X$ we define  the {\it $m$-mean curvature of $\partial E$ at $x$} as
\begin{equation}\label{defcurdefdef}H^m_{\partial E}(x):= \int_{X}  (\1_{X \setminus E}(y) - \1_E(y)) dm_x(y).\end{equation}
Observe that
\begin{equation}\label{defcur}H^m_{\partial E}(x) =  1 - 2 \int_E  dm_x(y).\end{equation}
 Having in mind \eqref{secondf021}, we have that,  if $\nu(E)<+\infty$,
$$\int_E H^m_{\partial E}(x) d\nu(x) = \int_E \left( 1 - 2 \int_E  dm_x(y) \right)  d\nu(x) = \nu(E) - 2\int_E\int_E dm_x(y) d\nu(x)$$ $$ = P_m(E) - \int_E\int_E dm_x(y) d\nu(x) = 2P_m(E) -\nu(E).$$ Consequently,
 \begin{equation}\label{1secondf021}\displaystyle \int_E H^m_{\partial E}(x) d\nu(x)=2P_m(E) -\nu(E).
\end{equation} and
\begin{equation}\label{pararm01}\frac{1}{\nu(E)}\int_\Omega H_{\partial E}^m(x)d\nu(x)=2\frac{P_m(E)}{\nu(E)}-1.
 \end{equation}

\subsection{Schwarz’s symmetrization}

Let $E \subset  \R^N$ be a measurable set of finite measure, and let $\1_E$ its characteristic
function. The {\it symmetric rearrangement} of $E$ is the ball $E^*$ centered at zero with $\vert E^* \vert  = \vert E \vert$,
i.e., with radius $\left(\frac{ \vert E \vert }{\omega_N} \right)^{\frac{1}{N}}$, where $\omega_N$  denotes the volume of the $N$-dimensional unit ball.
For a non-negative measurable function $f : \R^N \rightarrow \R$ vanishing at infinity, the {\it Schwarz’s
symmetrization} of $f$ is
$$f^*(x):=  \int_0^\infty \1_{\{ f > s \}^*}(x) ds,$$
where by definition, $(\1_E)^* = \1_{E^*}$. Thus, the level sets of $f^*$ are the rearrangements of the level sets $f$, implying the equimeasurability property
$$\vert \{ x \ : \ f^*(x) >s \} \vert = \vert \{ x \ : \ f(x) >s \} \vert.$$

The Schwarz’s symmetrization $f^*$ of a function $f$ inherits many measure geometric properties from
its source function $f$ (see \cite{Bandle}). It also fulfils some optimization properties with respect to integration. We will make  use of the following inequalities (see \cite{Liebloss}), the {\it Hardy-Littlewood’s inequality}:
\begin{equation}\label{HLineq}
\int_{\R^N} f_1(x) f_2(x) dx \leq \int_{\R^N} f^*_1(x) f^*_2(x) dx;
\end{equation}
and the {\it Riesz’s inequality}:
\begin{equation}\label{Rieszineq}
\int_{\R^N} f_1(x) \left( \int_{\R^N} f_2(x-y) f_3(y) dy \right) dx \leq \int_{\R^N} f^*_1(x) \left( \int_{\R^N} f^*_2(x-y) f^*_3(y) dy \right) dx
\end{equation}
We also need the general rearrangement inequality proved in~\cite{BLLuttinger}:
\begin{theorem}[see Theorem~3.8 in \cite{Liebloss}]\label{gririesz}
Let $m,k\in \mathbb{N}$, $m\ge k$, and $f_i$, $i=1,2,...,m$, nonnegative functions in $\mathbb{R}^N$, vanishing at infinity. Let $B$ a $k\times m$ matrix with coefficient $b_{ij}$ in the raw $i$ and column $j$. Then, if
$$I(f_1,f_2,...,f_m):=\int_{\mathbb{R}^N}\dots\int_{\mathbb{R}^N}\prod_{j=1}^m f_j\left(\sum_{i=1}^k b_{ij}x_i\right)dx_1\cdots dx_k,$$
we have that
$$I(f_1,f_2,...,f_m)\le  I(f_1^*,f_2^*,...,f_m^*),$$
where each $f_j^*$ is the symmetric-nonincreasing rearrangement of $f_j$.
\end{theorem}

\section{Torsional rigidity in  random walk spaces}

Let $[X,\mathcal{B},m,\nu]$ be a reversible random walk space.
  Given   $\Omega \in \mathcal{B}$,  we define  the {\it $m$-boundary of $\Omega$} by
$$\partial_m\Omega:=\{ x\in X\setminus \Omega \, : \, m_x(\Omega)>0 \}$$
and its {\it $m$-closure} as
$$\Omega_m:=\Omega\cup\partial_m\Omega.$$

  From now on we will assume that  $\Omega$ is $m$-connected (which imply that also  $\Omega_m$  is $m$-connected), $$0<\nu(\Omega)<\nu(\Omega_m)<\infty.$$

\begin{remark}\label{04012301}\rm
A first consequence of the above assumptions is that
\begin{equation}\label{04012302}
0<P_m(\Omega)<\nu(D).
\end{equation}
Indeed, if $P_m(\Omega)=0$ then , by~\eqref{secondf021}, $\displaystyle\int_\Omega m_x(\Omega)d\nu(x)=1,$
and consequently $m_x(\Omega)=1$ $\nu$-a.e. $x\in\Omega$. Therefore
$$L_m(\Omega_m\setminus\Omega,\Omega)=\int_\Omega m_x(\Omega_m\setminus\Omega)d\nu(x)=\int_\Omega (1-m_x(\Omega))d\nu(x)=0,$$
which contradicts tha $\Omega_m$ is $m$-connected (we are assuming $0<\nu(\Omega)<\nu(\Omega_m)$).

On the other hand, if $P_m(\Omega)=\nu(\Omega)$ then, by~\eqref{secondf021}, $m_x(\Omega)=0$ $\nu$-a.e. $x\in\Omega$. Therefore
$$L_m(\Omega,\Omega)=\int_\Omega m_x(\Omega)d\nu(x)=0,$$
which contradicts that $\Omega$ is $m$-connected. $\blacksquare$
\end{remark}

  Given    $p\geq 1$, we define
 $$L^p_0(\Omega_m, \nu):= \{ f \in L^p(\Omega_m, \nu) \ : \ f(x)=0 \ a.e. \ x \in \partial\Omega_m \}.$$

 We say that $\Omega$ satisfies  a {\it $p$-Poincar\'{e} inequality} if
there exists $\lambda >0$ such that
\begin{equation}\label{PoincareIneq2}
\lambda \int_\Omega \vert f (x) \vert^p d\nu(x)\leq \int_{\Omega_m \times \Omega_m} \vert \nabla f(x,y) \vert^p d(\nu\otimes m_x)(x,y)
\end{equation}
for all $f \in L^p_0(\Omega_m, \nu)$

Let us point out that the random walk spaces given  in Example~\ref{example.nonlocalJ}, for~$J$ with compact support, and in Example~\ref{example.graphs} satisfy a $2$-Poincar\'{e}'s type inequality, see~\cite{ElLibro, MSTBook}.

  In this section we will assume that $\Omega$ satisfies a $2$-Poincar\'{e} inequality.

As a consequence  of the results in \cite{ST0} (see also~\cite{MSTBook}), there is a unique solution  of the following homogenous Dirichlet problem for the $m$-Laplacian
\begin{equation}\label{torsioneq}
\left\{\begin{array}{ll}
-\Delta_m f_\Omega =1&\hbox{in } \Omega,\\[10pt]
f_\Omega =0&\hbox{on }\partial_m\Omega;
\end{array}\right.
\end{equation}
that is,
\begin{equation}\label{torsioneqti}
\left\{\begin{array}{ll}\displaystyle
-\int_{\Omega_m}   \left( f_\Omega(y)-f_\Omega(x)\right) dm_x(y)=1,
&  x\in  \Omega,
    \\ \\
f_\Omega(x)=0,&x\in \partial_m\Omega.
\end{array}\right.
\end{equation}
 We  denote by $f_\Omega$ this unique solution and  name it as the {\it $m$-stress function} of $\Omega$.  By the comparison principle given in~\cite{ST0}, we have that $f_\Omega \ge 0.$

\begin{definition}{\rm The {\it $m$-torsional rigidity of~$\Omega$}, $T_m(\Omega)$, is defined as the $L^1(\nu)$-norm of the torsion function:
$$T_m(\Omega)= \int_\Omega f_\Omega (x) d \nu(x).$$}
\end{definition}

 In the local case, it is well known  (see, for exmaple, \cite{BBP}) that
 $$T(B_R) = \frac{\omega_N}{N(N+2)}R^{N+2}.$$
 Then,
 $$T(B_R) \geq \vert B_R| \iff \frac{\omega_N}{N(N+2)}R^{N+2} \geq R^N \omega_N\iff R \geq \sqrt{N(N+2)}.$$
Contrary to  the local setting,
the $m$-torsional rigidity of $\Omega$ always satisfies
\begin{equation}\label{torsiomass}T_m(\Omega)\ge \nu(\Omega).\end{equation}
 Indeed, by the first equation in~\eqref{torsioneq}, for $x\in \Omega$, since  $m_x(\Omega_m)=1$, we have
\begin{equation}\label{recu01}f_\Omega(x)=1+\int_\Omega f_\Omega(y)dm_x(y),
\end{equation}
Hence
$$T_m(\Omega)= \int_\Omega f_\Omega (x) d \nu(x) = \nu(\Omega) + \int_\Omega \int_\Omega f_\Omega(y)dm_x(y) d \nu(x) \geq \nu(\Omega).$$
We will give   in Proposition~\ref{proptors01} a detailed description of $T_m(\Omega)$ by using a kind of geometrical terms relative to $\Omega$ via the random walk.

 The next result is the nonlocal version of equation \eqref{var1}. It is a particular case of Theorem~\ref{Charact1}.
 \begin{theorem}\label{$T_m$formula} We have
 \begin{equation}\label{thest01}\displaystyle T_m(\Omega)=\max_{
\hbox{\tiny$\begin{array}{c}g\in L^2(\Omega_m)\setminus\{0\}\\ g=0\hbox{ on }\partial_m\Omega
\end{array}$}
}\frac{\displaystyle\left(\int_\Omega gd\nu\right)^2}{\displaystyle \frac12\iint_{\Omega_m\times \Omega_m}|\nabla g(x,y)|^2dm_x(y)d\nu(x)},
\end{equation}
 and the maximum is attained at $f_\Omega$.
\end{theorem}

 In \cite{MSTBook} (see also~\cite{redbook}) we introduce the {\it spectral $m$-heat content} of $\Omega$ as
  $$  \mathbb{Q}_\Omega^m(t) =\int_\Omega v(t,x)d\nu(x),$$
 where $v(t,x)$ is the solution of
the {\it homogeneous Dirichlet problem for the $m$-heat equation}:
\begin{equation}\label{CPNL1dir}
\left\{ \begin{array}{ll} \displaystyle\frac{dv}{dt}(t,x) = \displaystyle\int_{\Omega_m} (v(t,y)   - v(t,x))dm_x(y), &(t,x)\in (0, +\infty)\times \Omega,
\\[12pt]
v(t,x)=0,&(t,x)\in(0, +\infty)\times\partial_m \Omega,
 \\[12pt]  u(0,x) =1,&x\in \Omega.\end{array}\right.
\end{equation}
Moreover,  we have (see~\cite{MSTBook} and~\cite{redbook}):
\begin{equation}\label{specheatcont01t01} \mathbb{Q}_\Omega^m(t)  = \sum_{k=0}^{+\infty} g_{m,\Omega}(k) \frac{e^{-t}t^k}{k!},
\end{equation}
where, for $k\in \mathbb{N}\cup\{0\}$,   $g_{m,\Omega}(k)$ is the measure of the amount of individuals that, starting in $\Omega$, end up in $\Omega$ after $k$ jumps without ever leaving $\Omega$, that is:
$$g_{m,\Omega}(0)=   \nu (\Omega)$$
and
$$g_{m,\Omega}(1)= \int_\Omega\int_\Omega dm_x (y)d\nu(x)  = L_m(\Omega, \Omega), $$
$$g_{m,\Omega}(2)= \int_\Omega\int_\Omega\int_\Omega dm_y(z)dm_x (y)d\nu(x),   $$
$$  \vdots$$
\begin{equation}\label{rem001}g_{m,\Omega}(n)= \int_{\hbox{\tiny$\underbrace{\Omega\times...\times\Omega}_n\times\Omega$}}dm_{x_n}(x_{n+1})\dots dm_{x_1}(x_2)d\nu(x_1).\end{equation}
i.e., $\mathbb{Q}_\Omega^m(t)$ is the expected value of the amount of individuals that start in $\Omega$ and end in $\Omega$ at time $t$ without ever leaving $\Omega$, when these individuals move by successively jumping according to $m$ and the number of jumps made up to time $t$ follows a Poisson distribution with rate~$t$.

\begin{lemma}\label{lemma-gmn02} We have that
\begin{equation}\label{gmn01}   \hbox{the sequence $\{g_{m,\Omega}(n) \, : \, n \in \N\}$ is non-increasing.}
\end{equation}
\end{lemma}

\begin{proof}
   For $n\ge 1$,
 $$\begin{array}{c}
 \displaystyle g_{m,\Omega}(n)= \int_{\hbox{\tiny$\underbrace{\Omega\times...\times\Omega}_n\times\Omega$}}dm_{x_n}(x_{n+1})\dots dm_{x_1}(x_2)d\nu(x_1)\\ \\
 \displaystyle =\int_{\hbox{\tiny$\underbrace{\Omega\times...\times\Omega}_{n-1}\times\Omega$}}m_{x_n}(\Omega)dm_{x_{n-1}}(x_{n})\dots dm_{x_1}(x_2)d\nu(x_1)\\ \\
 \displaystyle
 \le\int_{\hbox{\tiny$\underbrace{\Omega\times...\times\Omega}_{n-1}\times\Omega$}} dm_{x_{n-1}}(x_{n})\dots dm_{x_1}(x_2)d\nu(x_1)=g_{m,\Omega}(n-1).
 \end{array}$$
 Then   \eqref{gmn01} holds.
\end{proof}

\begin{remark}\label{rem35}\rm
  Observe that, by~\eqref{04012302}, we have $g_{m,\Omega}(1) < g_{m,\Omega}(0)$. We also have $$g_{m,\Omega}(2) < g_{m,\Omega}(1).$$
Indeed, using reversibility,
$$\begin{array}{l}
\displaystyle g_{m,\Omega}(2)= \int_\Omega\int_\Omega\int_\Omega dm_y(z)dm_x (y)d\nu(x)
\\ \\
 \displaystyle
=\int_X\int_X\int_X \1_\Omega(z)\1_\Omega(y)\1_\Omega(x)dm_y(z)dm_x (y)d\nu(x)
\\ \\
 \displaystyle
 =\int_X\int_X\int_X \1_\Omega(z)\1_\Omega(y)\1_\Omega(x)dm_x(z)dm_x (y)d\nu(x)
\\ \\
 \displaystyle
= \int_\Omega\int_\Omega   m_x(\Omega)dm_x (y)d\nu(x)
= \int_\Omega   \left(m_x(\Omega)\right)^2  d\nu(x)
 \\ \\
 \displaystyle
 \le  \int_\Omega m_x(\Omega)d\nu(x) = g_{m,\Omega}(1).
  \end{array} $$
  Then, if $g_{m,\Omega}(2) = g_{m,\Omega}(1)$, we have $$\displaystyle\int_\Omega m_x(\Omega)(1-m_x(\Omega))d\nu(x)=0.$$
Hence  $\Omega=A\cup B$,
  where $A:=\{x\in\Omega:m_x(\Omega)=0\}$ and  up to a $\nu$-null set, $B=\{x\in\Omega:m_x(\Omega)=1\}.$
 Now, we have
  $$L_m(A,B)=\int_A m_x(B)d\nu(x)=0,$$
  and consequenlty, since $\Omega$ is $m$-connected, $\nu(A)=0$ or $\nu(B)=0$, which yields a contradiction (remember Remark~\ref{04012301}).
  $\blacksquare$
\end{remark}

Let us now see  the nonlocal version of equation \eqref{exp6}.  Observe that the second statement in the next result gives a complete description of $T_m(\Omega)$ in term of the sequence of {\it probabilistic} terms $\{g_{m,\Omega}(n)  \, : \, n \in \N\}$.

 \begin{theorem}\label{proptors01} We have
 \begin{equation}\label{Nspecheatcont01t02} T_m(\Omega)=\int_0^\infty  \mathbb{Q}_\Omega^m(t) dt
\end{equation}
and
 \begin{equation}\label{specheatcont01t02} T_m(\Omega)=\sum_{k=0}^{+\infty} g_{m,\Omega}(k).
\end{equation}
\end{theorem}
\begin{proof} It is easy to see that if $v$ is the solution of the Dirichlet problem \eqref{CPNL1dir}, then  $$f(x):= \int_0^\infty v(x,t) dt$$
is the unique solution $f_\Omega$ of problem \eqref{torsioneq}. Hence, by Fubini's Theorem,
$$T_m(\Omega) =  \int_\Omega f (x) d \nu(x) = \int_\Omega \int_0^\infty v(x,t) dt d\nu(x) = \int_0^\infty  \mathbb{Q}_\Omega^m(t) dt.$$

 By \eqref{Nspecheatcont01t02} and \eqref{specheatcont01t01}, since the convergence in \eqref{specheatcont01t01} is uniform,  we have
 $$T_m(\Omega)=\int_0^\infty \sum_{k=0}^{+\infty} g_{m,\Omega}(k) \frac{e^{-t}t^k}{k!} dt = \sum_{k=0}^{+\infty} \int_0^\infty g_{m,\Omega}(k) \frac{e^{-t}t^k}{k!} dt = \sum_{k=0}^{+\infty} g_{m,\Omega}(k).$$
\end{proof}

 As consequence of \eqref{specheatcont01t02} we have the following result.
\begin{corollary} If $\Omega_1 \subset \Omega_2$, then $T_m(\Omega_1) \leq T_m(\Omega_2)$.
\end{corollary}

  Having in mind \eqref{exp5}, we give the following definition.

\begin{definition}{\rm We define  {\it the sequence of exit-$m$-moments of $\Omega$}
as
 $$EM^m_{j}(\Omega)=j\int_0^{+\infty}t^{j-1}\mathbb{Q}_\Omega^m(t) dt, \quad j\in \N.$$
 }
 \end{definition}
 Note that, as in~\eqref{exp3},
  $$EM^m_{1}(\Omega)=T_m(\Omega).$$

In the next result we also describe explicitly the sequence of exit-$m$-moments  in terms of the sequence $\{ g_{m,\Omega}(k) \, : \, k \in \N \}$.  In the context of Riemannian manifolds, see~\cite{CLM2017} for other type of expansions.
\begin{proposition}\label{proptors02} We have
 \begin{equation}\label{Nspecheatcont01t02mom} EM^m_{j}(\Omega)=j!\,
 \sum_{k=0}^{+\infty}
 \binom{k+j-1}{j-1} g_{m,\Omega}(k),\quad j=1,2,3,...
    \end{equation}
\end{proposition}
\begin{proof}
Let $j\ge 1$, then
$$EM^m_{j}(\Omega)=j\int_0^{+\infty}t^{j-1}\mathbb{Q}_\Omega^m(t)dt
=j\int_0^{+\infty}t^{j-1}\sum_{k=0}^{+\infty} g_{m,\Omega}(k) \frac{e^{-t}t^k}{k!}dt.$$
Now we can
 interchange the integral with the sum to get
$$\begin{array}{c}\displaystyle EM^m_j(\Omega)=j\,\sum_{k=0}^{+\infty}g_{m,\Omega}(k)\frac{1}{k!} \int_0^{+\infty} t^{j+k-1}e^{-t}dt\\ \\
\displaystyle=j\,\sum_{k=0}^{+\infty}g_{m,\Omega}(k)\frac{1}{k!} (k+j-1)!=j!\,
 \sum_{k=0}^{+\infty}
 \binom{k+j-1}{j-1} g_{m,\Omega}(k).
\end{array}$$
\end{proof}

Let us now define
\begin{equation}\label{Realyq}\lambda_{m,2}(\Omega)  = \inf_{\hbox{\tiny$\begin{array}{c}g\in L^2(\Omega_m)\setminus\{0\}\\ g=0\hbox{ on }\partial_m\Omega\end{array}$}}\frac{\displaystyle \frac12\int_{\Omega_m}\int_{\Omega_m} |\nabla g(x,y)|^2 dm_x(y)d\nu(x)}{\displaystyle\int_\Omega g(x)^2 d\nu(x) }.\end{equation}
 Since we are assuming   $\Omega$ satisfies  a $2$-Poincar\'{e} type inequality, we have $$\lambda_{m,2}(\Omega) >0.$$
And, since $\big||a|-|b|\big|\le |a-b|$  for all $a,b\in \mathbb{R}$,
\begin{equation}\label{deflmo}\lambda_{m,2}(\Omega) = \inf_{\hbox{\tiny$\begin{array}{c}g\in L^2(\Omega_m)\setminus\{0\}\\ g\ge 0\hbox{ on }\Omega\\g=0\hbox{ on }\partial_m\Omega\end{array}$}}\frac{\displaystyle \frac12\int_{\Omega_m}\int_{\Omega_m} |\nabla g(x,y)|^2 dm_x(y)d\nu(x)}{\displaystyle \int_\Omega g(x)^2 d\nu(x) }.
\end{equation}
 Similarly to the local case we have  the following  nonlocal version of~\eqref{lambdatorsionlocal1} (see Corollary~\ref{corlambdatorsionlocal1nlp} later on):
\begin{equation}\label{lambdatorsionlocal1nl}\lambda_{m,2}(\Omega)\le\frac{\nu(\Omega)}{T_m(\Omega)}.
\end{equation}

 We  also have that, see~\eqref{04012303}, \begin{equation}\label{da01}\frac{\nu(\Omega)}{T_m(\Omega)}\le\frac{P_m(\Omega)}{\nu(\Omega)}.
\end{equation}
Observe that, by~\eqref{04012302} we have that $\frac{P_m(\Omega)}{\nu(\Omega)}<1$. Therefore, from~\eqref{da01} and ~\eqref{lambdatorsionlocal1nl},
\begin{equation}\label{da02}
0<\lambda_{m,2}(\Omega)< 1.
\end{equation}

The following assumption will be used in the next result:
{\it There exists a non-null function
 $f\in L^2(\Omega_m,\nu)$    such that}
\begin{equation}\label{eigenproblem}\left\{\begin{array}{l}
\displaystyle-\int_{\Omega_m}(f(y)-f(x))dm_x(y)=\lambda_{m,2}(\Omega)f(x),\quad x\in\Omega,
\\ \\
f(x)=0,\quad x\in\partial_m\Omega.
\end{array}\right.
\end{equation}
  Observe that  then the infimum defining $\lambda_{m,2}(\Omega)$ in~\eqref{Realyq} is attained at $f$. We say that $\lambda_{m,2}(\Omega)$ is  the first eigenvalue of the $m$-Laplacian with homogeneous Dirichlet boundary conditions with associated eigenfunction $f$.   Note that, in fact, there is a non-negative eigenfunction associated to $\lambda_{m,2}(\Omega)$.

  In the next result we see that it is possible to obtain  $\lambda_{m,2}(\Omega)$ via the sequence $\{ g_{m,\Omega}(k) \, : \, k \in \N \}$ that characterize the torsional rigidity $T_m(\Omega)$ (Theorem~\ref{proptors01}) and the  exit-$m$-moments (Proposition \ref{proptors02}).

\begin{theorem}\label{fk01}  Assume $\lambda_{m,2}(\Omega)$ is an eigenvalue of   the $m$-Laplacian with homogeneous Dirichlet boundary conditions. Then:
\item{1.}
$
g_{m,\Omega}(n) > 0\quad\hbox{for all } n \in \N.
$

\item{ 2.} Assume moreover that there exists an eigenfunction $f$ associated to   $\lambda_{m,2}(\Omega)$ such that
\begin{equation}\label{condH}  \hbox{ $\alpha \le f\le \widetilde\alpha$ in $\Omega$, for some constants $\alpha,\widetilde\alpha>0$.}
\end{equation}
 Then,
\begin{equation}\label{pueq010}
\lambda_{m,2}(\Omega)=1-\displaystyle\lim_n\sqrt[n]{\frac{g_{m,\Omega}(2n)} {g_{m,\Omega}(n)}}.
\end{equation}
 \end{theorem}

\begin{proof}
  We have, for a non-negative (non-null) eigenfunction  $f$ associated to $\lambda_{m,2}(\Omega)$:
\begin{equation}\label{est001}
(1-\lambda_{m,2}(\Omega))f(x)=\int_\Omega f(y)dm_x(y),\quad x\in\Omega.
\end{equation}
Now, since $0<\lambda_{m,2}(\Omega)< 1$, we can write~\eqref{est001} as
$$f(x)=\frac{1}{1-\lambda_{m,2}(\Omega)}\int_\Omega f(y)dm_x(y).$$
Then, by induction, for $n\in \mathbb{N}$,
$$f(x_1)=\frac{1}{(1-\lambda_{m,2}(\Omega))^n}\int_{\Omega\times...\times\Omega}f(x_{n+1})dm_{x_n}(x_{n+1})\dots dm_{x_1}(x_2);$$
and, then, integrating over $\Omega$ with respect to $d\nu$, we have
\begin{equation}\label{expfn}0<\int_\Omega fd\nu=\frac{1}{(1-\lambda_{m,2}(\Omega))^n}\int_{\Omega\times...\times\Omega\times\Omega}f(x_{n+1})dm_{x_n}(x_{n+1})\dots dm_{x_1}(x_2)d\nu(x_1).\end{equation}

Let us see that
\begin{equation}\label{mund02}\displaystyle\int_{\Omega\times...\times\Omega\times\Omega}f(x_{n+1})^2dm_{x_n}(x_{n+1})\dots dm_{x_1}(x_2)d\nu(x_1)\le \int_\Omega f^2d\nu,
\end{equation}
In fact, by the reversibility of $\nu$ with respect to the random walk, for $n=1$, we have
$$\begin{array}{c}\displaystyle
\int_{\Omega\times\Omega}f(y)^2 dm_{x}(y)d\nu(x)
=\int_{X\times X}f(y)^2\1_\Omega(y)\1_\Omega(x) dm_{x}(y)d\nu(x)\\
\\
\displaystyle
=\int_{X\times X}f(x)^2\1_\Omega(x)\1_\Omega(y) dm_{x}(y)d\nu(x))
\displaystyle
=\int_{X}f(x)^2\1_\Omega(x)m_x(\Omega)   d\nu(x)\\
\\
\displaystyle\le\int_{X}f(x)^2\1_\Omega(x)   d\nu(x)=\int_\Omega   f^2d\nu.
\end{array}
$$
For $n=2$, using moreover Fubini's theorem,
$$\begin{array}{c}\displaystyle
\int_{\Omega\times\Omega\times\Omega}f(z)^2 dm_y(z)dm_{x}(y)d\nu(x)\\
\\
\displaystyle
=\int_{X\times X\times X}f(z)^2\1_\Omega(z)\1_\Omega(y)\1_\Omega(x) dm_y(z)dm_{x}(y)d\nu(x)\\
\\
\displaystyle
=\int_{X\times X\times X}f(z)^2\1_\Omega(z)\1_\Omega(x)\1_\Omega(y) dm_x(z)dm_{x}(y)d\nu(x)\\
\\
\displaystyle
=\int_{X\times X\times X}f(z)^2\1_\Omega(z)\1_\Omega(x)\1_\Omega(y)dm_{x}(y) dm_x(z)d\nu(x)
\\
\\ \displaystyle
=\int_{X\times X }f(z)^2\1_\Omega(z)\1_\Omega(x)m_x(\Omega)dm_{x}(y) dm_x(z)d\nu(x)
\\
\\ \displaystyle
\le\int_{X\times X}f(z)^2\1_\Omega(z)\1_\Omega(x)  dm_x(z)d\nu(x),
\end{array}
$$
and now we can use the case $n=1$. The general case follows by induction.

Then, by \eqref{expfn}, we have
$$\begin{array}{c}\displaystyle 0<\int_{\Omega\times...\times\Omega\times\Omega}f(x_{n+1})dm_{x_n}(x_{n+1})\dots dm_{x_1}(x_2)d\nu(x_1)\\
\\
\displaystyle \le \left(\int_{\Omega\times...\times\Omega\times\Omega}f(x_{n+1})^2dm_{x_n}(x_{n+1})\dots dm_{x_1}(x_2)d\nu(x_1)\right)^{1/2}g_{m,\Omega}(n)^{1/2}
\\
\\
\displaystyle \le \left(\int_\Omega f^2d\nu\right)^{1/2}g_{m,\Omega}(n)^{1/2};
  \end{array}$$
therefore, $g_{m,\Omega}(n)>0.$

\noindent  {\it Proof of 2.}
Dividing  the expression~\eqref{expfn} in $n$ between the one in $n+1$, we get
$$1=(1-\lambda_{m,2}(\Omega))\frac{\displaystyle\int_{\Omega\times...\times\Omega\times\Omega}f(x_{n+1})dm_{x_n}(x_{n+1})\dots dm_{x_1}(x_2)d\nu(x_1)}
{\displaystyle\int_{\Omega\times\Omega\times...\times\Omega\times\Omega}f(x_{n+2})dm_{x_{n+1}}(x_{n+2})dm_{x_n}(x_{n+1})\dots dm_{x_1}(x_2)d\nu(x_1)}.$$
Therefore,
$$1-\lambda_{m,2}(\Omega)=\frac{\displaystyle\int_{\Omega\times\Omega\times...\times\Omega\times\Omega}f(x_{n+2})dm_{x_{n+1}}(x_{n+2})dm_{x_n}(x_{n+1})\dots dm_{x_1}(x_2)d\nu(x_1)}
{\displaystyle\int_{\Omega\times...\times\Omega\times\Omega}f(x_{n+1})dm_{x_n}(x_{n+1})\dots dm_{x_1}(x_2)d\nu(x_1)},$$
or equivalently,
\begin{equation}\label{tl01}
1-\lambda_{m,2}(\Omega)=\frac{\tau_{n+1}\,g_{m,\Omega}(n+1)}{\tau_{n}\ g_{m,\Omega}(n)},
\end{equation}
where
\begin{equation}\label{media01}\tau_n=\frac{1}{g_{m,\Omega}(n)}\int_{\Omega\times...\times\Omega\times\Omega}f(x_{n+1})dm_{x_n}(x_{n+1})\dots dm_{x_1}(x_2)d\nu(x_1),
\end{equation}
  with $g_{m,\Omega}(n)$ given in~\eqref{rem001}, that is,
$$g_{m,\Omega}(n)=\int_{\Omega\times...\times\Omega\times\Omega} dm_{x_n}(x_{n+1})\dots dm_{x_1}(x_2)d\nu(x_1).$$
 Observe that  $\tau_n$ is the   average  of $g(x_1,x_2,...,x_n,x_{n+1}):=f(x_{n+1})$ in  $\Omega\times...\times\Omega\times\Omega$  with respect to the measure
$dm_{x_n}(x_{n+1})\dots dm_{x_1}(x_2)d\nu(x_1)$.
Since $0<\alpha\le f\le \widetilde\alpha $, we have
\begin{equation}\label{re003}\alpha\le \tau_n\le\widetilde{\alpha}.\end{equation}

 Now, from~\eqref{tl01}, we have that
\begin{equation}\label{re001}
(1-\lambda_{m,2}(\Omega))^n=\frac{\tau_{2n}\,g_{m,\Omega}(2n)}{\tau_{n}\ g_{m,\Omega}(n)}.
\end{equation}
Hence,
\begin{equation}\label{re002}
\log(1-\lambda_{m,2})=\frac{1}{n}\log\left(\frac{\tau_{2n}}{\tau_{n}}\right)+  \log\sqrt[n]{\frac{g_{m,\Omega}(2n)}{g_{m,\Omega}(n)}}.
\end{equation}
Since by~\eqref{re003}, $\displaystyle \lim_n\frac{1}{n}\log\left(\frac{\tau_{2n}}{\tau_{n}}\right)=0$, taking limits in \eqref{re002} we get~\eqref{pueq010}. \end{proof}

\begin{remark}\label{remdeejem}\rm \

\item{ 1.}  Let $[\R^N, d, m^J, \mathcal{L}^N]$  be the metric random walk space   given in Example~\ref{example.nonlocalJ} with $J$ continuous  and compactly supported. For $\Omega$ a bounded domain,
     the assumption~\eqref{condH}  is true,  see~\cite[Section 2.1.1]{ElLibro}.

\item{ 2.}
 For weighted  discrete graphs,   $\lambda_{m^G,2}(\Omega)$ is an eigenvalue with $0<\lambda_{m^G,2}(\Omega)\le 1$ (see~\cite{AGrigor}).    Now,   since we are assuming that $\Omega$ is $m^G$-connected,
 $0<\lambda_{m^G,2}(\Omega)< 1.$
And, by connectedness, using~\eqref{est001}, we have that~\eqref{condH}  is also true.

\item{ 3.}  Let us see what can happen if $\Omega$ is not $m^G$-connected.
Consider, for example, the weighted graph $G$ with five different vertices $V:=V(G) = \{x_1, x_2, x_3, x_4, x_5 \}$ and
$w_{x_i,x_{i+1}}=1$, for $i=1,2,3,4$,  and  $w_{x_i,x_j}=0$ otherwise.  We have,
$$m^G_{x_1} = \delta_{x_2}, m^G_{x_2} = \frac12 \delta_{x_1} + \frac12 \delta_{x_3}, m^G_{x_3} = \frac12 \delta_{x_2} + \frac12 \delta_{x_4},  m^G_{x_4} = \frac12 \delta_{x_3} + \frac12 \delta_{x_5},  m^G_{x_5} = \delta_{x_4},$$ $$ \nu^G= \delta_{x_1} +  2\delta_{x_2} +   2 \delta_{x_3} + 2\delta_{x_4} +\delta_{x_5}.$$

\item{ 3.1} Take $\Omega = \{ x_2,x_4 \}$, which is not $m^G$-connected.
It is easy to see that
$g_{m, \Omega}(n)=0$ for all $n\ge 1.$
And we have that $$T_{m^G}(\Omega) = \nu_G(\Omega)=\nu(\{x_2\})+\nu_G(\{x_4\})=T_{m^G}(\{x_2\})+T_{m^G}(\{x_4\}).$$

\item{ 3.2} Take now $\Omega:=\{x_1,x_2,x_4,x_5\}$, which is also not $m^G$-connected. In this case $g_{m^G,\Omega}(n)\neq 0$ for all $n\ge1$, and $$T_{m^G}(\Omega)   =T_{m^G}(\{x_1,x_2\})+T_{m^G}(\{x_4,x_5\})>\nu_{m^G}(\{x_1,x_2\})+\nu_{m^G}(\{x_4,x_5\})=\nu(\Omega).$$
Observe that   $\{x_1,x_2\}$ is $m^G$-connected,  and  $\{x_4,x_5\}$ is also  $m^G$-connected.

$\blacksquare$
\end{remark}

\section{The particular case of a nonlocal operator with non singular kernel}

In this section we study the particular case of the random walk space given in Example~\ref{example.nonlocalJ}, that is, we consider the metric measure space $(\R^N, d, \mathcal{L}^N)$, where $d$ is the Euclidean distance and $\mathcal{L}^N$ the Lebesgue measure on $\R^N$.  Let  $J:\R^N\to[0,+\infty[$ be a measurable, nonnegative and radially symmetric function  verifying  $\int_{\R^N}J(x)dx=1$. Let $m^J$ the random walk
$$m^J_x(A) = \int_A J(y-x) dy, \quad x \in \R^N,$$
for which the Lebesgue measure is reversible.

 We are going to prove a nonlocal version of the  Saint-Venant inequality. For this we need the following result.

\begin{lemma}\label{buenisimo}  Let $\Omega$ be a bounded domain in $\R^N$. If $J$ is radial and  non-increasing, then
\begin{equation}\label{despargo}
g_{m^J,\Omega}(k)\le g_{m^J,\Omega^*}(k)\quad\forall k\ge 0.
\end{equation}
\end{lemma}
\begin{proof}
It is obvious that
$$g_{m^J,\Omega}(0)=g_{m^J,\Omega^*}(0),$$
and, by Riesz inequality and having in mind that  $J^*=J$  and $(\1_\Omega)^* = \1_{\Omega^*}$, we have
$$g_{m^J,}\Omega(1)= \int_\Omega\int_\Omega J(x-y) dx = \int_{\R^N} \1_\Omega(x) \left( \int_{\R^N} J(x-y) \1_\Omega(y) dy \right) dx $$ $$ \leq \int_{\R^N} (\1_\Omega)^*(x) \left( \int_{\R^N} J^*(x-y) (\1_\Omega)^*(y) dy \right) dx  $$ $$=  \int_{\R^N} \1_{\Omega^*}(x) \left( \int_{\R^N} J(x-y)  \1_{\Omega^*}(y) dy \right) dx $$ $$ = \int_{\Omega^*}\int_{\Omega^*} J(x-y) dx =g_{\Omega^*}(1).$$

Let us now see that
$$g_{m^J,\Omega}(k)\le g_{m^J,\Omega^*}(k)\quad\forall k\ge 2.$$ Indeed, for $k=2$,
$$g_{m^J,\Omega}(2)=\int_{\mathbb{R}^N}\int_{\mathbb{R}^N}\int_{\mathbb{R}^N}\1_\Omega(x)\1_D(y)\1_\Omega(z)J(z-y)J(y-x)dxdydz.$$
Now, since
$$\left(\begin{array}{ccc}
x\ &y\ &z\end{array}\right)
    \cdot
\left(\begin{array}{ccccc}
1&0&0&0&-1\\
0&1&0&-1&1\\
0&0&1&1&0
\end{array}
\right)
=\left(\begin{array}{ccccc}
x\ &y\ &z\ &z\!-\!y\ &y\!-\!x\end{array}\right),$$
choosing the $3\times 5$ matrix
$$(b_{ij}):= \left(\begin{array}{ccccc}
1&0&0&0&-1\\
0&1&0&-1&1\\
0&0&1&1&0
\end{array}
\right),$$
we have
$$g_{m,\Omega}(2) = I(\1_\Omega, \1_\Omega, \1_\Omega, J,J).$$
Then, by Theorem~\ref{gririesz}, we have
$$\begin{array}{l}\displaystyle
g_{m^J,\Omega}(2) = I(\1_\Omega, \1_\Omega, \1_\Omega, J,J)\qquad\qquad\\ \\
\displaystyle \qquad\qquad \leq I((\1_\Omega)^*, (\1_\Omega)^*, (\1_\Omega)^*, J^*,J^*) = I(\1_{\Omega^*}, \1_{\Omega^*}, \1_{\Omega^*}, J,J) = g_{m^J,\Omega^*}(2).\end{array}$$
The inequalities for rest of $g_{m,\Omega}$(k) are obtained similarly.

\end{proof}

\begin{theorem}\label{simetr01} Let $\Omega$ be a bounded measurable subset of $\mathbb{R}^N$ and assume that  $J$ is radial and  non-increasing. Then,  we have the following inequalities:
\item[ 1.]  $\mathbb{Q}_\Omega^{m^J}(t) \le \mathbb{Q}_{\Omega^*}^{m^J}(t)\quad\forall t\ge 0.$
\item[ 2.] $T_{m^J}(\Omega)\le T_{m^J}(\Omega^*)$ \ \ {\it (Saint-Venant inequality)}.
\item[ 3.] $EM^{m^J}_{j}(\Omega)\le EM^{m^J}_{j}(\Omega^*)\quad\forall j\ge 1.$
\end{theorem}
\begin{proof} $1.$ It is consequence of \eqref{specheatcont01t01} and Lemma \ref{buenisimo}.

\noindent $2$. It is consequence  \eqref{specheatcont01t02} and Lemma \ref{buenisimo}.

\noindent $3$. It is consequence  of Proposition \ref{proptors02} and Lemma \ref{buenisimo}.
\end{proof}

\begin{remark} {\rm  A Faber-Krahn inequality
\begin{equation}\label{pueq02}\lambda_{m^J,2}(\Omega^*)\le \lambda_{m^J,2}(\Omega).
\end{equation}
can be obtained as a consequence of  \cite[Lemma A.2]{FS}.   Moreover, assuming that $J$ is   decreasing, and assuming also  $\lambda_{m^J,2}(\Omega)$ is an eigenvalue, or equivalently the infimum in the Rayleigh quotient
$$\lambda_{m^J,2}(\Omega) = \inf_{\hbox{\tiny$\begin{array}{c}g\in L^2(\Omega_m^J)\setminus\{0\}\\ g\ge 0\hbox{ on }\Omega\\g=0\hbox{ on }\partial_m\Omega\end{array}$}}\frac{\displaystyle \frac12\int_{\Omega_m^J}\int_{\Omega_m^J} |\nabla g(x,y)|^2dy dx}{\displaystyle \int_\Omega g(x)^2 dx } $$
is a minimum (we know this is true  for $J$ with compact support which, obviously, are not   decreasing),   by \cite[Lemma A.2]{FS}, one can also prove
$$\lambda_{m^J,2}(\Omega^*) =  \lambda_{m^J,2}(\Omega) \iff \Omega \ \hbox{is a ball}.$$
$\blacksquare$
}
\end{remark}

\subsection{Rescaling results} In this subsection we see that we can recover the local concepts and some of their properties from the nonlocal ones. In particular we give a different proof of the classical Saint-Venant inequality.

Set
\begin{equation}\label{fjai01} J_\epsilon(x):=\frac{1}{\epsilon^N}J\left(\frac{x}{\epsilon}\right),\quad\epsilon>0.
\end{equation}
And define
$$C_{J,2}=\frac{2}{\displaystyle \int_{\mathbb{R}^N}J(x)|x_N|^2dx.}$$
Observe that
$C_{J_\epsilon,2}=\frac{1}{\epsilon^2}C_{J,2}$.

\begin{theorem}\label{fkineq}  Let $\Omega$ be a bounded domain in $\R^N$. Assume $\displaystyle\int_{\mathbb{R}^N}J(x)|x|dx<+\infty$. We have:
$$\lim_{\epsilon\downarrow 0}\mathbb{Q}_\Omega^{m^{J_\epsilon}}\left(\frac{C_{J,2}}{\epsilon^2}t\right)=\mathbb{Q}_\Omega(t),$$
where $\mathbb{Q}_\Omega(t)$ is the (local) spectral heat content of $\Omega$;
and
\begin{equation}\label{resc001}
\lim_{\epsilon\downarrow 0}\frac{\epsilon^2}{C_{J,2}}T_{m^{J_\epsilon}}(\Omega)=T(\Omega).
\end{equation}
\end{theorem}
\begin{proof}
 The first part is consequence of the rescaling results proved in~\cite{ElLibro} (see also~\cite{redbook}) that also work if $\int_{\mathbb{R}^N}J(x)|x|dx<+\infty$ thanks to the general results given by A. Ponce in~\cite{Ponce}.  The second part is a consequence of the fact that  we can interchange the limit with the integral.
\end{proof}

By Theorems \ref{fkineq} and \ref{simetr01},  we can recover the classical {\it Saint-Venant inequality}:
\begin{theorem}[Saint-Venant inequality]\label{simetr02} Let $\Omega$ be a bounded domain in $\R^N$. Then,
$$T(\Omega)\le T(\Omega^*).$$
  And, more generally, for  any $j\ge 1$,
$$EM_j(\Omega)\le EM_j(\Omega^*).$$
\end{theorem}

\section{The particular case of a weighted graph}

 In this section we describe an iterative numerical method to get the torsional rigidity  of a non-trivial subset  of a weighted discrete graph. It is not our intention to give numerical results. We only want to show that~\eqref{rem001} and~\eqref{specheatcont01t02}  allow to use such iterative method.

Consider a weighted discrete graph $[V(G),\mathcal{B},m^G,\nu_G]$ as in Example~\ref{example.graphs} and a $\Omega$ a finite connected subset of $V(G)$.
Let us write $\Omega=\{x_1,x_2,...,x_N\}$ and $\partial_{m^{G}}\Omega=\{x_{N+1},...,x_M\}$, $x_i\ne x_j$ for $i\neq j$. Set $w_{ij}$ the weights between $x_i$ and $x_j$ (remember that $w_{ij}=0$ if $x_i\not\sim x_j$).

Set the weight of each $x_i\in\Omega$:
$$\displaystyle d_i=\sum_{j=1}^Mw_{ij},\  i=1,2,...,N.$$
Then, from~\eqref{rem001} and~\eqref{specheatcont01t02}, the following iterative scheme gives an approximation $T(n)$  of the torsion:
$$\left\{\begin{array}{ll}
\displaystyle T(0)=\sum_{i=1}^Nd_i, &\hbox{ (the term $g_{m^G,\Omega}(0)$)}
\\[16pt]
\displaystyle f^1_i=\sum_{j=1}^Nw_{i,j},\ i=1,2,...,N,
\\[16pt]
\displaystyle g(1)=\sum_{i=1}^Nf_i^1,& \hbox{ (the term $g_{m^G,\Omega}(1)$})
\\[16pt]
\displaystyle T(1)=T(0)+g(1),
\\[16pt]
\hbox{for } n\ge2:
\\[8pt]
\displaystyle \qquad f^{n}_i=\sum_{j=1}^N\frac{1}{d_j}f_j^{n-1}w_{i,j},\ i=1,2,...,N,
\\[16pt]
\displaystyle \qquad g(n)=\sum_{i=1}^Nf_i^n,& \hbox{ (the term $g_{m^G,\Omega}(n)$})
\\[16pt]
\displaystyle \qquad T(n)=T(n-1)+g(n).& \hbox{ ($\displaystyle \lim_nT(n)=T_{m^G}(\Omega)$})
\end{array}\right.$$

 From~\eqref{tl01}  we have that
$$\left|T_{m^G}(\Omega)-T(n)\right|= O\left((1-\lambda_{m,2}(\Omega))^{n+1}\right).$$

\section{The $m$-$p$-torsional rigidiy}

Brasco in \cite{Brasco1}, for $p >1$, defines the {\it $p$-torsional rigidity}  of the set $D$ as
$$T_p(D):=  \max_{v\in W^{1,2}_0(D)\setminus\{0\}}\frac{\displaystyle \left( \int_\Omega \vert v \vert dx \right)^{p}}{\displaystyle \int_D \vert\nabla v \vert^p dx}.$$

In \cite[Proposition 2.2]{Brasco1}, it is proved that
\begin{equation}\label{BasrcoR1}
T_p(D) = \left(\int_D v_D dx \right)^{p-1},
\end{equation}
where $v_D$ is the unique weak solution of the problem
\begin{equation}\label{torsioneqforplocal}
\left\{\begin{array}{ll}
-\Delta_p v_D =1&\hbox{in } D,\\[10pt]
v_D(x)=0&\hbox{on }\partial D.
\end{array}\right.
\end{equation}

Now we are going to get the nonlocal version of equation \eqref{BasrcoR1}.

In this section we will we assume that   $1 < p < \infty$, $\Omega \in \mathcal{B}$,  $0 < \nu(\Omega) < \nu(X)$ and $\Omega$ satisfies a $p$-Poincar\'{e} inequality (see~\eqref{PoincareIneq2}).

 From the reversibility of $\nu$ respect to $m$, we have the following {\it integration by parts formula}
 \begin{equation}\label{INTBYPart} \begin{array}{ll}
 -\displaystyle\int_{\Omega_m \times \Omega_m} \vert \nabla f(x,y) \vert^{p-2} \nabla f(x,y) g(x) dm_x(y) d\nu(x) \\[10pt] = \displaystyle\frac12 \int_{\Omega_m \times \Omega_m} \vert \nabla f(x,y) \vert^{p-2} \nabla f(x,y) \nabla g(x,y) dm_x(y) d\nu(x), \end{array}
 \end{equation}
 if $f, g \in L^p(\Omega_m, \nu)$.

We give the following definition of the homogeneous Dirichlet problem for the $m$-$p$-Laplacian.

\begin{definition}{\rm   Given $g \in L^1(\Omega, \nu)$, we say that $f \in L^p_0(\Omega_m, \nu)$ is a solution of problem
\begin{equation}\label{Dirichletp}
\left\{\begin{array}{ll}
-\Delta_{m,p} f =g&\hbox{in } \Omega,\\[10pt]
f(x)=0&\hbox{on }\partial_m\Omega;
\end{array}\right.
\end{equation}
if it verifies
$$\left\{\begin{array}{ll}
-\hbox{div}_m(|\nabla f|^{p-2}\nabla f)(x)=g&\hbox{in } \Omega,\\[10pt]
f(x)=0&\hbox{on }\partial_m\Omega;
\end{array}\right.$$
that is,
$$
\left\{\begin{array}{ll}\displaystyle
-\int_{\Omega_m}   |f(y)-f(x)|^{p-2}(f(y)-f(x)) dm_x(y)= g(x),
&  x\in  \Omega,
    \\ \\
f(x)=0,&x\in \partial_m\Omega.
\end{array}\right.
$$
}
\end{definition}

  Existence and  uniqueness  are  given in \cite{ST0} (see also~\cite{MSTBook}). Nevertheless, and for the sake of completeness, we give  the next result with a  different proof.

\begin{theorem}\label{EUp}
There is a unique solution $f_{\Omega,p} \geq 0$  of the  homogenous Dirichlet problem for the $m$-$p$-Laplacian,
\begin{equation}\label{mpLaplace}
\left\{\begin{array}{ll}\displaystyle
-\int_{\Omega_m}   |f_{\Omega,p}(y)-f_{\Omega,p}(x)|^{p-2}(f_{\Omega,p}(y)-f_{\Omega,p}(x)) dm_x(y)= 1,
&  x\in  \Omega,
    \\ \\
f_{\Omega,p}(x)=0,&x\in \partial_m\Omega.
\end{array}\right.
\end{equation}
Moreover, $f_{\Omega,p}$ is the only minimizer of the variational problem
\begin{equation}\label{VarPro2}
\min_{f \in  L_0^2(\Omega_m, \nu) \setminus \{ 0 \}} \mathcal{F}_{m,p}(f),
\end{equation}
where
$$\mathcal{F}_{m,p}(f):= \frac{1}{2p} \int_{\Omega_m \times \Omega_m} \vert \nabla f(x,y) \vert^p d(\nu\otimes m_x)(x,y) - \int_{\Omega} f(x) d\nu(x).$$
And,
\begin{equation}\label{igualdaB}
\int_{\Omega}  f_{\Omega,p}(x) d\nu(x) = \frac{1}{2} \int_{\Omega_m \times \Omega_m} \vert \nabla  f_{\Omega,p}(x,y) \vert^p d(\nu\otimes m_x)(x,y).
\end{equation}
\end{theorem}
\begin{proof}    First note that $\mathcal{F}_{m,p}$  is convex and lower semicontinuous in $L^p(\Omega,\nu)$, thus weakly lower semicontinuous (see~\cite[Corollary 3.9]{BrezisAF}).
Set $$\theta := \inf_{f \in L_0^2(\Omega_m, \nu) \setminus \{ 0 \}} \mathcal{F}_{m,p}(f),$$ and let $\{ f_n \}$ be a minimizing sequence. Then,
$$\theta = \lim_{n \to \infty} \mathcal{F}_{m,p}(f_n)  \quad \hbox{and} \quad K:= \sup_{n \in \NN} \mathcal{F}_{m,p}(f_n)  < + \infty\,.$$
Since $\Omega_m$ satisfies a the Poincar\'{e} inequality \eqref{PoincareIneq2}, by Young's inequality, we have
$$\displaystyle \lambda\int_\Omega \left|f_n(x) \right|^p \, {d\nu(x)} \leq \int_{\Omega_m \times \Omega_m} \vert \nabla f_n(x,y) \vert^p d(\nu\otimes m_x)(x,y)$$
$$= 2p \mathcal{F}_{m,p}(f_n) + \int_{\Omega} f_n(x) d \nu(x) \leq 2pK + \int_\Omega \vert f_n (x)  \vert d\nu(x) $$ $$ \leq 2pK + \frac{\lambda}{2} \int_\Omega \left|f_n(x) \right|^p \, {d\nu(x)} + \left(\frac{\lambda}{2} \right)^{-\frac{1}{p-1}} \nu(\Omega).$$
Therefore,  we obtain that
$$\int_{\Omega} |f_n(x)|^p \, {d\nu(x)} \leq C \quad \forall n \in \NN.$$
Hence, up to a subsequence, we have
$$f_n \rightharpoonup f_{\Omega,p} \ \ \ \hbox{in }   L_0^p(\Omega_m, \nu).$$
Furthermore, using the weak lower semicontinuity of the functional $\mathcal{F}_{m,p}$, we get
$$\mathcal{F}_{m,p}(f_{\Omega,p}) = \inf_{f \in L_0^2(\Omega_m, \nu) \setminus \{ 0 \}} \mathcal{F}_{m,p}(f).$$
 Since the functional $\mathcal{F}_{m,p}$ is strictly convex, we have that $f_{\Omega,p}$ is the unique minimizer, and since $ \mathcal{F}_{m,p}(\vert f \vert ) \leq  \mathcal{F}_{m,p}(f)$, we have that $f_{\Omega,p} \geq 0$.

Thus, given $\lambda >0$ and $w \in L_0^p(\Omega,\nu)$ , we have
$$0 \leq \frac{\mathcal{F}_{m,p}(f_{\Omega,p} + \lambda w)- \mathcal{F}_{m,p}(f_{\Omega,p})}{\lambda} $$
or, equivalently,
$$0\leq \frac{1}{\lambda} \Big[ \frac{1}{2p} \int_{\Omega_m \times \Omega_m} \vert \nabla (f_{\Omega,p} + \lambda w)(x,y) \vert^p d(\nu\otimes m_x)(x,y) - \int_{\Omega} (f_{\Omega,p} + \lambda w)(x) d\nu(x) $$ $$ - \left( \frac{1}{2p} \int_{\Omega_m \times \Omega_m} \vert \nabla f_{\Omega,p}(x,y) \vert^p d(\nu\otimes m_x)(x,y) - \int_{\Omega} f_{\Omega,p}(x) d\nu(x) \right)\Big].$$
Now, since $p >1$, we pass to the limit as $\lambda \downarrow 0$ to obtain
$$0 \leq \frac12\int_{\Omega_m \times \Omega_m} \vert \nabla f_{\Omega,p}(x,y) \vert^{p-2}  \nabla f_{\Omega,p}(x,y) \nabla w(x,y) d(\nu\otimes m_x)(x,y) - \int_\Omega w(x) d \nu(x).$$

Taking $\lambda <0$ and proceeding as above we obtain the opposite inequality. Consequently, we conclude that
$$0 = \frac12\int_{\Omega_m \times \Omega_m} \vert \nabla f_{\Omega,p}(x,y) \vert^{p-2}  \nabla f_{\Omega,p}(x,y) \nabla w(x,y) d(\nu\otimes m_x)(x,y) - \int_\Omega w(x) d \nu(x)$$ $$= -\int_{\Omega_m}\int_{\Omega_m} \vert \nabla f_{\Omega,p}(x,y) \vert^{p-2}  \nabla f_{\Omega,p}(x,y)  dm_x(y)w(x)d\nu(x) - \int_\Omega w(x) d \nu(x),$$
which shows that $f_{\Omega,p}$ is  solution of \eqref{mpLaplace}.

 Finally, taking $w=f_{\Omega,p}$ in the above first equation we get~\eqref{igualdaB}.
\end{proof}

\begin{definition}{\rm  We call to $f_{\Omega,p}$ as the {\it $p$-torsional function} of $\Omega$,   and we define the {\it $m$-$p$-torsional rigidity} of $\Omega$ as
$$T_{m,p}(\Omega):= \left(\int_\Omega f_{\Omega,p}  d \nu \right)^{p-1}.$$

Note that $T_{m}(\Omega) = T_{m,2}(\Omega)$.
}
\end{definition}

\begin{theorem}\label{Charact1}
 We have
\begin{equation}\label{NonBasrcoR1}
T_{m,p}(\Omega) =  \max_{g\in L_0^p(\Omega_m, \nu) \setminus \{ 0 \}}
\frac{\displaystyle\left(\int_\Omega |g|d\nu \right)^p}{\displaystyle  \frac{1}{2}\int_{\Omega_m\times \Omega_m}|\nabla g(x,y)|^p \, d(\nu\otimes m_x)(x,y)},
\end{equation}
and  the maximum is attained at $f_{\Omega,p}$.
\end{theorem}
\begin{proof}  By~\eqref{igualdaB},
\begin{equation}\label{eq01paratN}
\frac{1}{2}\int_{\Omega_m}\int_{\Omega_m}|\nabla f_{\Omega,p}(x,y)|^p \, d(\nu\otimes m_x)(x,y)=\int_\Omega f_{\Omega,p}(x) d\nu(x).
\end{equation}
Therefore,
$$\displaystyle T_{m,p}(\Omega)=\left(\int_\Omega f_{\Omega,p} d\nu \right)^{p-1} = \frac{\displaystyle\left(\int_\Omega f_{\Omega,p} d\nu\right)^p}{\displaystyle \frac12\int_{\Omega_m\times \Omega_m}|\nabla f_\Omega(x,y)|^pdm_x(y)d\nu(x)}.$$

Let $g\in L_0^p(\Omega_m, \nu)$, $g \not=0$. Since $f_{\Omega,p}$ is a solution of Problem~\eqref{mpLaplace},
$$\int_\Omega |g| d\nu =\frac12 \int_{\Omega_m}\int_{\Omega_m} \vert \nabla f_{\Omega,p}(x,y) \vert^{p-2}  \nabla f_{\Omega,p}(x,y) \nabla |g|(x,y) d(\nu\otimes m_x)(x,y).$$
Then, by H\"{o}lder's inequality,
$$ \begin{array}{c}\displaystyle\int_\Omega |g| d\nu \le \frac{1}{2}\left(\int_{\Omega_m}\int_{\Omega_m}\vert \nabla f_{\Omega,p}(x,y) \vert^{p} \, d(\nu\otimes m_x)(x,y)\right)^{1/p'}
\\ \\ \displaystyle
\times\left(\int_{\Omega_m}\int_{\Omega_m}\vert \nabla g(x,y) \vert^p \, d(\nu\otimes m_x)(x,y)\right)^{1/p}.
\end{array}$$
Then, from \eqref{igualdaB},
$$\int_\Omega |g| d\nu \le \frac{1}{2}\left(2\int_\Omega f_{\Omega,p}(x) d\nu(x)\right)^{1/p'}\left(\int_{\Omega_m}\int_{\Omega_m}\vert \nabla g(x,y) \vert^p \, d(\nu\otimes m_x)(x,y)\right)^{1/p}$$
$$=\frac{1}{2^{1/p}} \left( T_{m,p}(\Omega) \right)^{1/p}\left(\int_{\Omega_m}\int_{\Omega_m}\vert \nabla g(x,y) \vert^p \, d(\nu\otimes m_x)(x,y)\right)^{1/p}.$$
Thus,
$$\displaystyle T_{m,p}(\Omega)
\geq \frac{\displaystyle\left(\int_\Omega |g| d\nu\right)^p}{\displaystyle \frac12\int_{\Omega_m\times \Omega_m}|\nabla g(x,y)|^pdm_x(y)d\nu(x)},$$
and consequently \eqref{NonBasrcoR1} holds.
\end{proof}

 We now define, for $p\ge 1$,
\begin{equation}\label{RayleighC}
\begin{array}{c}\displaystyle
\lambda_{m,p}(\Omega):= \inf_{f \in  L_0^p(\Omega_m, \nu) \setminus \{ 0 \}}  \frac{ \displaystyle\frac{1}{2} \displaystyle\int_{\Omega_m \times \Omega_m} \vert \nabla f(x,y) \vert^p d(\nu\otimes m_x)(x,y)}{\displaystyle\int_\Omega \vert f (x) \vert^p d\nu(x)}
\\ \\
\displaystyle
=\inf_{\hbox{\tiny$\begin{array}{c}f \in  L_0^p(\Omega_m, \nu) \setminus \{ 0 \}\\ f\ge 0\hbox{ on }\Omega\end{array}$}}  \frac{ \displaystyle\frac{1}{2} \displaystyle\int_{\Omega_m \times \Omega_m} \vert \nabla f(x,y) \vert^p d(\nu\otimes m_x)(x,y)}{\displaystyle\int_\Omega \vert f (x) \vert^p d\nu(x)}.
\end{array}
\end{equation}

As a consequence of the above result we have:

\begin{corollary}\label{corlambdatorsionlocal1nlp}   For $p>1$  we have
\begin{equation}\label{lambdatorsionlocal1nlp}\lambda_{m,p}(\Omega)\le \frac{\nu(\Omega)^{p-1}}{T_{m,p}(\Omega)}.
\end{equation}
\end{corollary}

\begin{proof} By Theorem \ref{Charact1}, we have
$$
\frac{1}{T_{m,p}(\Omega)} =
\frac{\displaystyle  \frac{1}{2}\int_{\Omega_m\times \Omega_m}|\nabla  f_{\Omega,p}(x,y)|^p \, d(\nu\otimes m_x)(x,y)}{\displaystyle\left(\int_\Omega f_{\Omega,p}d\nu \right)^p}.
$$
Now
$$\int_\Omega f_{\Omega,p}(x) d\nu(x) \leq \left(\int_\Omega (f_{\Omega,p}(x))^p d\nu(x)\right)^{\frac1p} \nu(\Omega)^{\frac{1}{p'}}.$$
Hence
$$\frac{1}{T_{m,p}(\Omega)} \geq  \frac{\displaystyle  \frac{1}{2}\int_{\Omega_m}\int_{\Omega_m} |\nabla f_{\Omega,p}(x,y)|^p d(\nu\otimes m_x)(x,y)}{\nu(\Omega)^{p-1}\displaystyle \int_\Omega (f_{\Omega,p}(x))^p d\nu(x) } \ge \frac{\lambda_{m,p}(\Omega)}{\nu(\Omega)^{p-1}}.$$
\end{proof}

 Fusco, Maggi and Pratelli in \cite{FMP1} (see also \cite{Avinyo}, \cite{FMP2} and \cite{PratelliSaracco}) generalized the classical concept of Cheeger constant, introducing, for $p \geq \frac{N-1}{N}$, the {\it  $p$-Cheeger constant} of and open set $\Omega \subset \R^N$ of finite measure as
$$h_p(\Omega):= \inf \left\{ \frac{P(E)}{\vert E \vert^p} : E \subset \Omega \ \hbox{is open} \right\}.$$
Note that for $h_1(\Omega)$ is the classical Cheeger constant.

In \cite{MST2} (see also~\cite{MSTBook}), for a set $\Omega \in \mathcal{B}$  such that $0 < \nu(\Omega) < \nu(X)$, we define its {\it $m$-Cheeger constant} as
 $$h_1^m(\Omega):= \inf \left\{ \frac{P_m(E)}{\nu(E)}   :   E \in \mathcal{B}, \ E \subset \Omega, \ \nu(E) >0 \right\}$$
 and  we prove (see~\cite[Theorem 3.37]{MSTBook}) that
\begin{equation}\label{RayleighCns01}\lambda_{m,1}(\Omega)=h_1^m(\Omega).
\end{equation}

\begin{remark}\rm  For any $p\ge 1$,
\begin{equation}\label{29d001}\lambda_{m,p}(\Omega)\le h_1^m(\Omega)=\lambda_{m,1}(\Omega).
\end{equation}
Indeed, for any $E\subset \Omega$, $\nu(E)>0$,
$$\lambda_{m,p}(\Omega)  \le \frac{\displaystyle \frac12\int_{\Omega_m}\int_{\Omega_m} |\nabla \1_E(x,y)|^p dm_x(y)d\nu(x)}{\displaystyle \int_\Omega \1_E(x) d\nu(x) }
$$ $$= \frac{\displaystyle\frac12\int_{\Omega_m}\int_{\Omega_m} |\nabla \1_E(x,y)|  dm_x(y)d\nu(x)}{\displaystyle \int_\Omega \1_E(x)  d\nu(x) }
=\frac{P_m(E)}{\nu(E)}.$$
Then taking infimum, and on account of~\eqref{RayleighCns01}, we get~\eqref{29d001}. $\blacksquare$
\end{remark}

Now, we introduce the following nonlocal version of the  $p$-Cheeger constant.

\begin{definition}{\rm  Let $p >1$, we define its {\it $m$-$p$-Cheeger constant} of~$\Omega$ as
 $$h_p^m(\Omega):= \inf \left\{ \frac{P_m(E)}{\nu(E)^p}   :   E \in \mathcal{B}, \ E \subset \Omega, \ \nu(E) >0 \right\},$$
}
\end{definition}

 Similarly to the local case (see, for example,~\cite[Proposition 5.2]{BBPrinari}), we have the following relation between the Chegeer constants and the $m$-$p$-torsional rigidity.

 \begin{theorem}   For $p >1$  we have
 \begin{equation}\label{ChegeerT1}
  2^{p-1}\frac{h_1^m(\Omega){}^p}{\nu(\Omega_m)^{p-1}} \leq  \frac{1} {T_{m,p}(\Omega)}
  \leq h_p^m(\Omega),
 \end{equation}
 and
  \begin{equation}\label{ChegeerT2}
 \lim_{p \to 1^+} \frac{1}{T_{m,p}(\Omega)}=\lim_{p\to 1^+}h_p^m(\Omega)= h_1^m(\Omega).
 \end{equation}
 \end{theorem}
 \begin{proof} By the coarea formula \eqref{coarea} and Cavalieri's principle, we have
 $$\displaystyle \frac12\int_{\Omega_m\times \Omega_m}|\nabla f_{\Omega,p}(x,y)|d(\nu\otimes m_x)(x,y) = \int_{0}^{+\infty} P_m( \{  x\in \Omega: f_{\Omega,p}(x) >t \}) dt$$ $$  \geq \int_{0}^{+\infty} h_1^m(\Omega) \nu(\{ x\in \Omega: f_{\Omega,p}(x) >t \}) dt $$ $$= h_1^m(\Omega) \int_{0}^{+\infty}  \nu(\{ x\in \Omega: f_{\Omega,p} (x) >t \}) dt =  h_1^m(\Omega)\int_\Omega f_{\Omega,p}(x) d \nu(x) . $$
 Hence, by H\"older's inequality and \eqref{igualdaB}, we obtain
 $$h_1^m(\Omega) \leq \frac{\displaystyle\frac12\int_{\Omega_m\times \Omega_m}|\nabla f_{\Omega,p}(x,y)|d(\nu\otimes m_x)(x,y)}{\displaystyle\int_\Omega f_{\Omega,p}(x) d \nu(x)} $$ $$ \leq
  \frac{\displaystyle\frac12 \left(\int_{\Omega_m\times \Omega_m}|\nabla f_{\Omega,p}(x,y)|^p d(\nu\otimes m_x)(x,y)\right)^{\frac{1}{p}}  \nu(\Omega_m)^{\frac{p-1}{p}}}{\displaystyle\int_\Omega f_{\Omega,p}(x) d \nu(x)} $$ $$= \frac{1}{2^{\frac{p-1}{p}}} \frac{\left(\displaystyle\int_\Omega f_{\Omega,p}(x) d \nu(x)\right)^{\frac{1}{p}}  \nu(\Omega_m)^{ \frac{p-1}{p}}}{\displaystyle\int_\Omega f_{\Omega,p}(x) d \nu(x)}$$
  $$ =  \frac{1}{2^{\frac{p-1}{p}}} \frac{  \nu(\Omega_m)^{\frac{p-1}{p}}}{\left(\displaystyle\int_\Omega f_{\Omega,p}(x) d \nu(x)\right)^{\frac{p-1}{p}}} = \frac{1}{2^{\frac{p-1}{p}}}\left( \frac{ \nu(\Omega_m)^{p-1}}{T_{m,p}(\Omega)} \right)^{\frac{1}{p}},$$
  and, from here
  \begin{equation}\label{ChegeerT1laprime}
  2^{p-1}\frac{h_1^m(\Omega){}^p}{\nu(\Omega_m)^{p-1}} \leq  \frac{1} {T_{m,p}(\Omega)}.
  \end{equation}
   On the other hand, by \eqref{NonBasrcoR1}, for any $E \in \mathcal{B}, \ E \subset \Omega, \ \nu(E) >0$,  we have
 $$
\frac{1}{T_{m,p}(\Omega) } \leq
\frac{\displaystyle \frac{1}{2}\int_{\Omega_m\times \Omega_m}|\nabla \1_E(x,y)|^p \, d(\nu\otimes m_x)(x,y)}{\displaystyle\left(\int_\Omega \1_E d\nu \right)^p} = \frac{P_m(E)}{\nu(E){}^p},$$
from where,
\begin{equation}\label{ChegeerT1lasegun}
     \frac{1} {T_{m,p}(\Omega)}
  \leq h_p^m(\Omega).
 \end{equation}
 And~\eqref{ChegeerT1} is proved.

Taking limits in~\eqref{ChegeerT1}, we have
  \begin{equation}\label{siiprevm01}h_1^m(\Omega) \leq  \liminf_{p \to 1^+} \frac{1}{T_{m,p}(\Omega)}\le \liminf_{p\to 1^+}h_p^m(\Omega),
  \end{equation}
  and
  $$  \limsup_{p \to 1^+} \frac{1}{T_{m,p}(\Omega)}\le \limsup_{p\to 1^+}h_p^m(\Omega).$$
 Let us now see  that
\begin{equation}\label{sii}  \limsup_{p\to 1+}h_p^m(\Omega) \leq h_1^m(\Omega).
\end{equation}
Indeed, for any $E \in \mathcal{B}, \ E \subset \Omega, \ \nu(E) >0$,  we have
$$
 h_p^m(\Omega)  \leq
 \frac{P_m(E)}{\nu(E)^p},$$
and, from here
$$\limsup_{p\to 1+}h_p^m(\Omega) \leq \frac{P_m(E)}{\nu(E)},$$
which allows to prove~\eqref{sii}. Finally,~\eqref{siiprevm01} and~\eqref{sii} gives~\eqref{ChegeerT2}.
\end{proof}

  P\'{o}lya \cite{P2} proves that, among all bounded open and convex planar sets, the following inequality holds
\begin{equation}\label{PolyaInq1}
\frac13 \leq \frac{T(D) P(D)^2}{\vert D \vert^3},
\end{equation}
 being the constant $\frac13$ optimal. This was generalized in~\cite{BBPrinari0} to dimension $N\ge 3$.
On the other hand, Makai \cite{Makai} proves that, among all bounded open and convex planar sets, the following upper bound holds
  \begin{equation}\label{MakaiInq}
 \frac{T(D) P(D)^2}{\vert D \vert^3} \leq \frac23,
\end{equation}
  being the constant $\frac23$ optimal.  See~ \cite{BBPrinari0}  for a conjecture in   dimension $N\ge 3$.
 Estimates \eqref{PolyaInq1} and \eqref{MakaiInq} are generalized for the $p$-Laplacian by Fragala,  Gazzola and Lamboley in \cite{FGL}.

Recall that $\Omega$ is $m$-calibrable if  $ h_1^m(\Omega)= \frac{P_m(\Omega)}{\nu(\Omega)}.$

\begin{corollary} We have
\begin{equation}\label{Polya2}
\frac{\nu(\Omega)^2}{ P_m(\Omega)}\leq{T_{m}(\Omega)}.
\end{equation}
Moreover, if $\Omega$ is $m$-calibrable, then
\begin{equation}\label{Polya1}
T_m(\Omega) \leq \frac12\frac{\nu(\Omega)^2\nu(\Omega_m)}{P_m(\Omega)^2}.
\end{equation}
 \end{corollary}

 \begin{proof} Taking $p=2$ in~\eqref{ChegeerT1}, since $T_{m,2}(\Omega) = T_{m}(\Omega)$, we have
  \begin{equation}\label{inf2}  2 \frac{h_1^m(\Omega){}^2}{\nu(\Omega_m)} \le\frac{1} {T_{m}(\Omega)}
  \leq h_2^m(\Omega) .
  \end{equation}
 Then, since $h_2^m(\Omega) \leq \frac{P_m(\Omega)}{\nu(\Omega)^2}$, from the second inequality in~\eqref{inf2} we get
 $$\frac{1} {T_{m}(\Omega)}
  \leq \frac{P_m(\Omega)}{\nu(\Omega)^2},$$ and~\eqref{Polya2} holds.
 On the other hand, assuming that $\Omega$ is $m$-calibrable,   we have $h_1^m(\Omega)= \frac{P_m(\Omega)}{\nu(\Omega)}$, and, substituting this value in the first inequality of~\eqref{inf2}, we have
 $$ 2\frac{P_m(\Omega)^2}{\nu(\Omega)^2\nu(\Omega_m)} \leq  \frac{1} {T_{m}(\Omega)},$$
 from  where \eqref{Polya1} holds.
 \end{proof}

Observe that, from~\eqref{04012302},  \eqref{lambdatorsionlocal1nl} and~\eqref{Polya2}, we have
 \begin{equation}\label{04012303}\nu(\Omega)<\frac{\nu(\Omega)^2}{ P_m(\Omega)}\leq{T_{m}(\Omega)} \le \frac{\nu(\Omega)}{\lambda_{m,2}(\Omega)}.
 \end{equation}
In the next example we will see that the second and third inequalities in \eqref{04012303}   are  sharp. We see that they are equalities for the most simple connected set for  weighted discrete graphs, which is trivially $m^G$-calibrable.

\begin{example}\label{7en01}\rm
\item{ 1.} Consider the weighted discrete lasso graph $V(G)=\{x,y\}$ with weights $w_{xx}=a>0$, $w_{xy}=b>0$ and $w_{yy}=0$ (we are in a situation of Example~\ref{example.graphs}). And take $\Omega=\{x\}$, which is $m^G$-connected (because of the loop). It is easy to see that
$$\nu_G(\Omega)=a+b,$$
$$P_{m^G}(\Omega)=b,$$
$$T_{m^G}(\Omega)=\frac{(a+b)^2}{b},$$
and $$\lambda_{m^G,2}(\Omega)=\frac{b}{a+b}.$$
Hence,
$$ \frac{\nu_G(\Omega)^2}{ P_{m^G}(\Omega)}={T_{m^G}(\Omega)}=\frac{\nu_G(\Omega)}{\lambda_{m^G,2}(\Omega)}.$$
 \item{ 2.} For the weighted discrete graph  $V(G)=\{x,y_1,y_2,\dots,y_k\}$, $k\ge 2$ with weights $w_{xx}=a>0$, $w_{xy_i}=b_i>0$ and $w_{y_iy_j}=0$ for any $i,j$, if we set $b=\sum_{i=1}^kb_j$, and take $\Omega=\{x\}$, we have the same results   than for the lasso graph.
$\blacksquare$
\end{example}

   In the next result we will see the influence of the $m$-mean curvature of $\Omega$.
 Observe first that, by \eqref{pararm01},
 \begin{equation}\label{pararm01NN}\frac{\displaystyle 1+\frac{1}{\nu(\Omega)}\int_\Omega H_{\partial\Omega}^m(x)d\nu(x)}{2}=\frac{P_m(\Omega)}{\nu(\Omega)}.
 \end{equation}
Then,~\eqref{Polya2} is equivalent to
\begin{equation}\label{equivtopolya2}\frac{\displaystyle1+\frac{1}{\nu(\Omega)}\int_\Omega H_{\partial\Omega}^m(x)d\nu(x)}{2}\frac{\nu(\Omega)^3}{ P_m(\Omega)^2}\leq{T_{m}(\Omega)}.
\end{equation}
Remember  also that
$$-1\le \frac{1}{\nu(\Omega)}\int_\Omega H_{\partial\Omega}^m(x)d\nu(x)\le 1.$$
Then, as  an inmediate consequence of~\eqref{equivtopolya2} we have:

 \begin{corollary}\label{PMIneq}
 Assume that b there exists $\beta \in \R$ such thatb
 \begin{equation}\label{lapol002dd}
 -1<\beta\le \frac{1}{\nu(\Omega)}\int_\Omega H_{\partial\Omega}^m(x)d\nu(x)<1.
\end{equation}
Then
\begin{equation}\label{la3001}
\left(\frac{\beta+1}{2} \right)\frac{\nu(\Omega)^3}{ P_m(\Omega)^2}\leq{T_{m}(\Omega)}.
\end{equation}
 \end{corollary}

 By~\eqref{pararm01}, we  have
 $$\frac{1}{\nu(\Omega)}\int_\Omega H_{\partial\Omega}^m(x)d\nu(x)\le \alpha<1\ \Leftrightarrow\ \frac{P_m(\Omega)}{\nu(\Omega)}\le \frac{\alpha+1}{2}.$$
Now, since~\eqref{Polya2} can be written as
$$T_{m}(\Omega)\ge \frac{\nu(\Omega)^2}{ P_m(\Omega)}=\frac{\nu(\Omega)}{ P_m(\Omega)}\nu(\Omega),$$
we obtain  the following result.
 \begin{corollary}\label{Corlapol001dxcur}
  Assume that  there exists $\alpha \in \R$ such that
 \begin{equation}\label{lapol001dxcur} \displaystyle -1<\frac{1}{\nu(\Omega)}\int_\Omega H_{\partial\Omega}^m(x)d\nu(x)\le \alpha <1.
\end{equation}
  Then
 $$T_{m}(\Omega)\ge   \frac{2}{\alpha+1}\nu(\Omega).$$
\end{corollary}

\begin{remark}\rm

\item{1.}   Let us  remark  that,  assuming \eqref{lapol001dxcur},  by the above Corollary and  by~\eqref{lambdatorsionlocal1nl}, we have
$$ \lambda_{m,2}(\Omega) \le \frac{\alpha+1}{2}.$$

\item{2.} Observe that~\eqref{la3001} is a  P\'{o}lya-type inequality for   subsets satisfying~\eqref{lapol002dd}; and that~\eqref{Polya1} is a  Makai-type inequality for calibrable subsets.

\item{3.}    As a consequence of~\eqref{equivtopolya2} and~\eqref{Polya1}, if $\Omega$ is calibrable then
    $$\frac{1}{\nu(\Omega)}\int_\Omega H_{\partial\Omega}^m(x)d\nu(x)\le\frac{\nu(\partial_m\Omega)}{\nu(\Omega)},$$
    or equivalently, using~\eqref{pararm01},
    $$h_1^m(\Omega)=\frac{P_m(\Omega)}{\nu(\Omega)}\le \frac12\left(1+\frac{\nu(\partial_m\Omega)}{\nu(\Omega)}\right).$$
$\blacksquare$
\end{remark}

We have the following result (see~\cite{kafri} in the local case).

\begin{theorem} We have,
\begin{equation}\label{29d002} \left(\frac{\lambda_{m,1}(\Omega)}{p} \right)^p \leq \lambda_{m,p}(\Omega)\le \lambda_{m,1}.
\end{equation}
 And consequently,
 \begin{equation}\label{ChegeerT2dd01}
  \lim_{p\to 1^+}\lambda_{m,p}(\Omega) = \lambda_{m,1}= h_1^m(\Omega).
 \end{equation}
\end{theorem}
\begin{proof}
 The second inequality of~\eqref{29d002}  is given in~\eqref{29d001}.
  On the other hand, for $p >1$, we have, for any $a,b\in \mathbb{R}$,
 $$\vert \vert b \vert^{p-1} b - \vert a \vert^{p-1} a \vert \leq p\vert b - a\vert \max\left\{\vert b \vert^{p-1},\vert a \vert^{p-1}\right\}.$$
 Hence
 $$\vert \nabla (\vert u \vert^{p-1} u)(x,y) \vert \leq p\vert \nabla u(x,y)\vert\max\left\{ \vert u(y) \vert^{p-1},\vert u(x) \vert^{p-1}\right\},$$
 and consequently, for $u \in  L_0^p(\Omega_m, \nu) \setminus \{ 0 \}$,  we have
 \begin{equation}\label{mund03}
  \begin{array}{c}
  \displaystyle\lambda_{m,1}(\Omega) \leq  \frac{ \displaystyle\frac{1}{2} \displaystyle\int_{\Omega_m \times \Omega_m} \vert \nabla (\vert u \vert^{p-1} u)(x,y) \vert d(\nu\otimes m_x)(x,y)}{\displaystyle\int_\Omega \vert u(x)\vert^p   d\nu(x)}
   \\ \\
   \displaystyle   \leq  \frac{ \displaystyle\frac{p}{2} \displaystyle\int_{\Omega_m \times \Omega_m} \vert \nabla u(x,y) \vert \max\left\{ \vert u(y) \vert^{p-1},\vert u(x) \vert^{p-1}\right\}d(\nu\otimes m_x)(x,y)}{\displaystyle\int_\Omega \vert u(x)\vert^p   d\nu(x)}.
   \end{array}
\end{equation}

 We claim now that
 \begin{equation}\label{max} \begin{array}{c}
    \displaystyle \int_{\Omega_m \times \Omega_m} \vert \nabla u(x,y) \vert \max\left\{ \vert u(y) \vert^{p-1},\vert u(x) \vert^{p-1}\right\}d(\nu\otimes m_x)(x,y)\\ \\
   \displaystyle =  2\int_{\Omega_m \times \Omega_m}  \vert \nabla u(x,y) \vert\,\1_{\{(x,y)\in\Omega_m\times\Omega_m:u(x)> u(y)\}}(x,y) \vert u(x) \vert^{p-1}d(\nu\otimes m_x)(x,y) .\end{array}
   \end{equation}
   Indeed,  by the reversibility of $\nu$ respect to $m$, and having in mind that $\nabla u(x,y) =0$ if $u(x) = u(y)$ and $\vert \nabla u(x,y) \vert = \vert \nabla u(y,x) \vert$, we have
 $$ \displaystyle \int_{\Omega_m \times \Omega_m} \vert \nabla u(x,y) \vert \max\left\{ \vert u(y) \vert^{p-1},\vert u(x) \vert^{p-1}\right\}d(\nu\otimes m_x)(x,y) $$ $$= \int_{\Omega_m \times \Omega_m}  \vert \nabla u(x,y) \vert\,\1_{\{(x,y)\in\Omega_m\times\Omega_m:u(x)> u(y)\}}(x,y) \vert u(x) \vert^{p-1}d(\nu\otimes m_x)(x,y)
 $$
 $$+\int_{\Omega_m \times \Omega_m}  \vert \nabla u(x,y) \vert\,\1_{\{(x,y)\in\Omega_m\times\Omega_m:u(y)> u(x)\}}(x,y) \vert u(y) \vert^{p-1}d(\nu\otimes m_x)(x,y)
$$
 $$ = \int_{\Omega_m \times \Omega_m}  \vert \nabla u(x,y) \vert\,\1_{\{(x,y)\in\Omega_m\times\Omega_m:u(x)> u(y)\}}(x,y) \vert u(x) \vert^{p-1}d(\nu\otimes m_x)(x,y)
 $$
 $$ +\int_{\Omega_m \times \Omega_m}  \vert \nabla u(x,y) \vert\,\1_{\{(x,y)\in\Omega_m\times\Omega_m:u(x)> u(y)\}}(x,y) \vert u(x) \vert^{p-1}d(\nu\otimes m_x)(x,y)
 $$
 $$=  \displaystyle  2\int_{\Omega_m \times \Omega_m}  \vert \nabla u(x,y) \vert\,\1_{\{(x,y)\in\Omega_m\times\Omega_m:u(x)> u(y)\}}(x,y) \vert u(x) \vert^{p-1}d(\nu\otimes m_x)(x,y).
 $$

 Now, applying H\"older's inequality, we get
 $$   \int_{\Omega_m \times \Omega_m} \vert \nabla u(x,y) \vert\,\1_{\{(x,y)\in\Omega_m\times\Omega_m:u(x)> u(y)\}} (x,y) \vert u(x) \vert^{p-1}d(\nu\otimes m_x)(x,y)
 $$
 $$
 \begin{array}{c}
    \displaystyle \le \left(\int_{\Omega_m \times \Omega_m}  \vert \nabla u(x,y) \vert^p \,\1_{\{(x,y)\in\Omega_m\times\Omega_m:u(x)> u(y)\}}(x,y)d(\nu\otimes m_x)(x,y) \right)^{\frac{1}{p}}
    \\[6pt]
    \displaystyle
    \hfill \times\left(\displaystyle\int_{\Omega_m\times\Omega_m} \vert u(x) \vert^{p} d\nu\otimes m_x(y) \right)^{\frac{p-1}{p}}
    \end{array}
 $$
  $$
    \begin{array}{c}
    \displaystyle \le \left(\int_{\Omega_m \times \Omega_m}   \vert \nabla u(x,y) \vert^p \,\1_{\{(x,y)\in\Omega_m\times\Omega_m:u(x)> u(y)\}}(x,y)d(\nu\otimes m_x)(x,y) \right)^{\frac{1}{p}}
    \\[6pt]
    \displaystyle
    \hfill \times\left(\displaystyle\int_\Omega \vert u(x) \vert^{p} d \nu(x) \right)^{\frac{p-1}{p}}
    \end{array}
 $$
  $$ =\left(\int_{\Omega_m \times \Omega_m} \frac12\vert \nabla u(x,y) \vert^p (x,y)d(\nu\otimes m_x)(x,y) \right)^{\frac{1}{p}}
     \left(\displaystyle\int_\Omega \vert u(x) \vert^{p} d \nu(x) \right)^{\frac{p-1}{p}},
 $$
 where   reversibility is used,  as in the proof of \eqref{max}, to get the last equality. Then
 $$ \begin{array}{c}
    \displaystyle \int_{\Omega_m \times \Omega_m} \vert \nabla u(x,y) \vert \max\left\{ \vert u(y) \vert^{p-1},\vert u(x) \vert^{p-1}\right\}d(\nu\otimes m_x)(x,y)\\ \\ \displaystyle
    \le 2\left(\int_{\Omega_m \times \Omega_m} \frac12\vert \nabla u(x,y) \vert^p (x,y)d(\nu\otimes m_x)(x,y) \right)^{\frac{1}{p}}
     \left(\displaystyle\int_\Omega \vert u(x) \vert^{p} d \nu(x) \right)^{\frac{p-1}{p}}
     \end{array}.$$

 Hence, using the above inequality, from~\eqref{mund03} we get
 $$\lambda_{m,1}(\Omega) \leq   \frac{\displaystyle p  \left(\displaystyle\frac{1}{2}\int_{\Omega_m \times \Omega_m} \vert \nabla u(x,y) \vert^p d(\nu\otimes m_x)(x,y) \right)^{\frac{1}{p}} \left(\displaystyle\int_\Omega \vert u(x) \vert^{p} d \nu(x) \right)^{\frac{p-1}{p}}}{\displaystyle\int_\Omega \vert u(x)\vert^p   d\nu(x)}$$
 $$= \frac{\displaystyle p \left(\displaystyle \frac{1}{2} \int_{\Omega_m \times \Omega_m} \vert \nabla u(x,y) \vert^p d(\nu\otimes m_x)(x,y) \right)^{\frac{1}{p}}}{\displaystyle\left(\int_\Omega \vert u(x)\vert^p   d\nu(x) \right)^{\frac{1}{p}}}.$$
 Thus
 $$  \left(\frac{\lambda_{m,1}(\Omega)}{p} \right)^p \leq \frac{ \displaystyle\frac12 \int_{\Omega_m \times \Omega_m} \vert \nabla u(x,y) \vert^p d(\nu\otimes m_x)(x,y) }{\displaystyle\int_\Omega \vert u(x)\vert^p   d\nu(x) }.$$
 The, taking infimum in $u \in  L_0^p(\Omega_m, \nu) \setminus \{ 0 \}$, we obtain that
$$ \left(\frac{\lambda_{m,1}(\Omega)}{p} \right)^p \leq \lambda_{m,p}(\Omega),$$
and~\eqref{29d002} is proved.
Finally,~\eqref{ChegeerT2dd01} is a direct consequence of~\eqref{29d002} and~\eqref{RayleighCns01}.
\end{proof}

\subsection{A rescaling result}
Set $J_\epsilon$ as in~\eqref{fjai01}. Define
$$C_J=\frac{2}{\displaystyle \int_{\mathbb{R}^N}J(x)|x_N|dx}.$$
Observe that $C_{J_\epsilon}=\frac{1}{\epsilon}C_J$.

 If $\displaystyle\int_{\mathbb{R}^N}J(x)|x|dx<+\infty$,  we have  (see~\cite{redbook}):
\begin{equation}\label{limitT}\lim_{\epsilon\downarrow 0} \frac{C_J}{\epsilon}h_{1}^{m^{J_\epsilon}}(\Omega)=h_1(\Omega).
\end{equation}
Remember that by~\eqref{resc001},
$$\lim_{\epsilon\downarrow 0}\frac{\epsilon^2}{C_{J,2}}T_{m^{J_\epsilon}}(\Omega)=T(\Omega).$$
Now,  from~\eqref{inf2},
  $$  \frac{h_1^{m^{J_\epsilon}}(\Omega){}^2}{|\Omega_{m^{J_\epsilon}}|} \leq \frac12  \frac{1} {T_{{m^{J_\epsilon}}}(\Omega)}.$$
  Then
  $$ \frac{\displaystyle \frac{1}{\epsilon^2}h_1^{m^{J_\epsilon}}(\Omega){}^2}{|\Omega_{m^{J_\epsilon}}|} \leq \frac12  \frac{1} {\epsilon^2T_{{m^{J_\epsilon}}}(\Omega)},$$
  and, taking limits as $\epsilon\to 0$,
\begin{equation}\label{nb02} \frac{ h_1 (\Omega){}^2}{|\Omega|} \leq \frac{C_J{}^2}{2C_{J,2}} \frac{1} { T(\Omega)}.
\end{equation}
Observe that $$\frac{C_J{}^2}{2C_{J,2}}=  \frac{\displaystyle  \int_{\mathbb{R}^N}J(x)|x_N|^2dx}{\left(\displaystyle\int_{\mathbb{R}^N}J(x)|x_N|dx\right)^2}\ge 1.$$ But, $\frac{C_J{}^2}{2C_{J,2}}$ is as close to 1 as we want by choosing  adequately $J$. So we can get
\begin{equation}\label{nb02wc01} \frac{ h_1 (\Omega){}^2}{|\Omega|} \leq   \frac{1} { T(\Omega)},
\end{equation}
and  in particular, for $\Omega$  calibrable we get the Makai-type inequality
\begin{equation}\label{quasiMakaiInq}
 \frac{T(\Omega) P(\Omega)^2}{\vert \Omega \vert^3} \leq 1.
\end{equation}

\section{Torsional rigidity on Quantum Graphs as a $m$-torsional rigidity on graphs}

 Torsional rigidity on quantum graphs was introduce by  Colladay, Kaganovskiy  and McDonald in~\cite{CKM2017}. To the best of our knowledge, after this paper, the only existing literature on this topic is the paper by Mugnolo and Plumer~\cite{MP}, where    the torsional rigidity of a quantum graph is related to the   rigidity of an associated weighted combinatorial graph. We will interpret here that result with the  (nonlocal) rigidity of a weighted graph.

 Let $\mathcal{G}$ be a compact, finite, connected quantum graph.  Let $V$ be the set of vertices of $\mathcal{G}$ and $E$ be the set of edges. Fora vertex $x \in V$,  let ${\rm deg}_{\mathcal{G}}(x)$ denote is {\it degree}, i.e. the number of edges incident in $x$. We suppose that $\mathcal{G}$ has at least one vertex of degree $1$. Set
 $$V_D:= \{ x \in V  : {\rm deg}_{\mathcal{G}}(x) = 1 \}$$
 and set $V_N:= V \setminus V_D$. We assume that the graph does not contain multiple edges between the same vertices but it can contain  at most one loop at each vertex (we comment on this later on).  Let us call  $\ell_e$ or $\ell_{x,y}$ the length of the edge $e$ that join the vertices $x$ and $y$.

For each $e \in E$ there exists an increasing an bijective function
$$
	\begin{array}{rlcc}
 c_e:&e&\to& [0,\ell_e]\\
 &x&\rightsquigarrow& x_{e},
 \end{array}
$$  $x_{e}$ is called the coordinate of the point $x\in e$.

A function $u$ on a metric graph $\mathcal{G}$ is a collection of functions $[u]_{e}$
defined on
$[0,\ell_{e}]$ for all $e\in E.$	Throughout this work, $ \int_{\mathcal{G}} u(x)  dx$  denotes
$ \sum_{e\in E} \int_{0}^{\ell_{e}} [u]_{e}(x_e)\, dx_e$.

 For $A \subset \mathcal{G}$, the {\it length} of $A$ is defined as
$$\ell(A) = \int_{\mathcal{G}} \1_A dx.$$

  Let $\Delta_{\mathcal{G}}$ the Laplacian on $\mathcal{G}$ with homogeneous Dirichlet boundary condition at vertices in $V_D$ and with the Kirchhoff type condition on the vertices in $V_N$, that is, its associated quadratic form $a_{\mathcal{G}}$ is given by
 $$a_{\mathcal{G}}(u):= \int_{\mathcal{G}} \vert u'(x) \vert^2 dx = \sum_{e \in E} \vert u'(x_e) \vert^2 dx_e$$
 on the domain
 $$H_{\mathcal{G}}(\mathcal{G},V_D):= \left\{ u = (u_e)_{e \in E} \in \bigoplus_{e \in E} H^1(0,l_e)   :   u(v) = 0 \hbox{ for } v \in V_D, \ u \ \hbox{continuos in }    V_N \right\}.$$

  Let $v$ be  the solution of
\begin{equation}\label{eqparaquantumgraph}
\left\{\begin{array}{ll}
-\Delta_\mathcal{G} v(x)=1,&x\in \mathcal{G},\\[10pt]
v(x)=0,&x\in V_D.
\end{array}\right.
\end{equation}
The function~$v$ is called the {\it torsion function} of $\mathcal{G}$, and the {\it (quantum)   torsional rigidity} of~$V_D$ is given by the $L^1$-norm of $v$:
$$T_q(\mathcal{G}):= \int_\mathcal{G} \vert v \vert dx.$$

In \cite{MP} Mugnolo and Plumer show that, if $v$ is the torsion function of $\mathcal{G}$, then  $f=2v_{|_V} : V \rightarrow \R$ is the unique solution of  the following problem:
\begin{equation}\label{torsioneqgraphs}
\left\{\begin{array}{ll}\displaystyle
-\frac{1}{\displaystyle \sum_{y\sim x}\ell_{yx}}\sum_{y\sim x}\frac{1}{\ell_{yx}}\left(f(y)-f(x)\right)=1,& x\in V_N,
    \\ \\
f(x)=0,&x\in V_D.
\end{array}\right.
\end{equation}
 And they prove that
\begin{equation}\label{ftrf01}
T_q(\mathcal{G})=\frac{1}{12}\sum_{e \in E}\ell_e^3+\frac12\sum_{x\in V}\left(\sum_{y\sim x}\ell_{yx} +\ell_{xx} \right)v(x)
.
\end{equation}
Observe that in the above expression, $\displaystyle\sum_{y\sim x}\ell_{yx} +\ell_{xx} =
\sum_{y\sim x,\, y\neq x}\ell_{yx} +2\ell_{xx} $. If we had $k (\ge 2)$ loops at the vertex $x$ with lengths $\ell_{xx}(i)$, $i=1,2,...k$, then we should change $\displaystyle 2\ell_{xx} $ by $ 2(\ell_{xx}(1)+...+\ell_{xx}(k)).$

   Take  $c>0$ large enough such that (we do not mark the dependence on $c$)
\begin{equation}\label{choiceofc} \widetilde{w}_{xx}:= \sum_{y\sim x}c\ell_{yx}-\sum_{y\sim x}\frac{1}{c\ell_{yx}}  > 0   \quad\forall x\in V_N
\end{equation}
Observe that, since  $\mathcal{G}$ is finite, such a $c$ exists.

Let us consider the  weighted  graph $G_c$ having the same vertices and edges than $\mathcal{G}$ with weights (we do not mark the dependence on $c$ in $\widetilde w_{xx}$):
$$\begin{array}{l}w_{yx}=\frac{1}{c\ell_{yx}}\quad \hbox{for }   y\sim x,\ y\neq x,\\ \\
w_{xx}=\frac{1}{c\ell_{xx}}+ c\ell_{xx}+\widetilde{w}_{xx}\quad{if }\ \ell_{xx}\neq 0,
\\ \\
w_{xx}=\widetilde{w}_{xx}\quad{if }\ \ell_{xx}= 0.
\end{array}
$$
On account of~\eqref{choiceofc}, we have that
\begin{equation}\label{onacc01}\sum_{y\sim x}w_{yx}=\sum_{y\sim x}c\ell_{yx}  + c\ell_{xx}.
\end{equation}
And, then,   {   from~\eqref{torsioneqgraphs}, we have that
$f_c:=2c^2v_{|_V}$ satisfies
\begin{equation}\label{torsioneqconc}
\left\{\begin{array}{ll}\displaystyle
 -\frac{1}{\displaystyle \sum_{y\sim x}w_{yx}}\sum_{y\sim x}w_{yx}\left(f_c(y)-f_c(x)\right)=1,& x\in V_N,
    \\ \\
f_c(x)=0,&x\in V_D.
\end{array}\right.
\end{equation}
Observe that, since $V_D=\partial_{m^{G_c}}V_N$, $f_c$ is solution of the problem
$$\left\{\begin{array}{ll}
-\Delta_{m^{G_c}} f_c =1&\hbox{in } \Omega,\\[10pt]
f_c=0&\hbox{on }\partial_{m^{G_c}}\Omega;
\end{array}\right. $$

Then we have that formula~\eqref{ftrf01} given in~\cite{MP} can be written using weighted discrete graphs, seen as random walk spaces, as follows.
\begin{theorem}\label{Charact1}   We have
\begin{equation}\label{iiets01} T_q(\mathcal{G}) =  \frac{1}{12}\sum_{e \in E}\ell_e^3+\frac{1}{4}\frac{1}{c^3}T_{m^{G_c}}(V_N),
\end{equation}
whatever  $c$ is chosen in~\eqref{choiceofc}.
\end{theorem}
\begin{proof}Indeed, from~\eqref{onacc01},
$$\begin{array}{c}
\displaystyle  T_{m^{G_c}}(V_N)=\sum_{x\in V}\left(\sum_{y\sim x}w_{xy}\right)f_c(x)
\\ \\   \displaystyle =2c^2\sum_{x\in V}\left(\displaystyle \sum_{y\sim x}c\ell_{yx} + c\ell_{xx} \right)  v(x)
=2c^3\sum_{x\in V}\left(\displaystyle \sum_{y\sim x}\ell_{yx}  + \ell_{xx} \right)  v (x).
\end{array}$$
And hence the statement~\eqref{iiets01} follows from~\eqref{ftrf01}.  \end{proof}

 As a consequence of the above theorem and \eqref{Polya2}  we recover the equivalent to Proposition 4.8 of~\cite{MP}.
\begin{corollary}
We have, for any  $c>0$ satisfying~\eqref{choiceofc},
\begin{equation}\label{IneqIT} T_q(\mathcal{G})  \geq  \frac{1}{12}\sum_{e \in E}\ell_e^3 +\frac{1}{4}\frac{1}{c^3} \frac{\nu_{G_c}(V_N)^2}{P_{m^{G_c}}(V_N)}.
\end{equation}

\end{corollary}

\begin{remark}\rm
\item{ 1.} Observe that if we assume that $\ell_e =1$ for all edge $e$ in $\mathcal{G}$, and we   have not loops,
\begin{equation}\label{MGboundal01}
T_q(\mathcal{G})  \geq  \frac{1}{12} \sharp(E) + \frac{1}{4}\frac{\left( \sum_{x \in V_N} {\rm deg}_{\mathcal{G}}(x) \right)^2}{\sum_{x \in V_N}\sharp( \{ y \in V_D \, : \, y \sim  x \})}   \ge \frac{1}{12} \sharp(E) + \frac{1}{4}\sum_{x \in V_N} {\rm deg}_{\mathcal{G}}(x).\end{equation}
Indeed,
$\nu_{G_c}(V_N) = c \sum_{x \in V_N} {\rm deg}_{\mathcal{G}}(x)$ and
$ P_{m^{G_c}}(V_N) = \frac{1}{c} \sum_{x \in V_N} \sharp( \{ y \in V_D  :  y \sim  x \}).$
Then, the first inequality in~\eqref{MGboundal01} follows from~\eqref{IneqIT},   and the second inequality follows since, for each $x\in V_N$, ${\rm deg}_{\mathcal{G}}(x) \ge \sharp( \{ y \in V_D  :  y \sim  x \})$.

\item{ 2.}  Consider a star metric graph $\mathcal{G}$, with Dirichlet conditions imposed on all vertices except the central one, and with a possible loop in the central vertex. Suppose that there are $k$ Dirichlet vertices with their edges joining the central vertex having  length $\ell_i$, $i=1,2,...,k$, and the possible loop at the central vertex with length $\ell_0\ge 0$ (if $\ell_0=0$ we do not have a loop and we have only a star). Then, on account of Theorem~\ref{Charact1} and Example~\ref{7en01}, for $c$ satisfying~\eqref{choiceofc}, we have that
$$\begin{array}{ll}\displaystyle T_q(\mathcal{G})&\displaystyle =\frac{1}{12}\sum_{i=0}^k\ell_i^3+\frac{1}{4c^3}\frac{\left(2c\ell_0+c\sum_{i=1}^k\ell_i\right)^2}{\sum_{i=1}^k\frac{1}{c\ell_i}}
\\ \\&\displaystyle =\frac{1}{12}\sum_{i=0}^k\ell_i^3+\frac{1}{4}\frac{\left(2\ell_0+\sum_{i=1}^k\ell_i\right)^2}{\sum_{i=1}^k\frac{1}{\ell_i}}.
\end{array}
$$
 The above equality recover, as could not be otherwise, the result of Example~3.10 of~\cite{MP}. We see that in this case that we have equality in~\eqref{IneqIT} (this is also remarked in~\cite[Proposition 4.8]{MP}).
In the particular case that $\ell_i=1$ for $i=1,2,...,k$ and $\ell_0=0$, then
$T_q(\mathcal{G}) = \frac13 k$, and all the inequalities in~\eqref{MGboundal01} are equalities.
$\blacksquare$
\end{remark}

 \

\noindent {\bf Acknowledgments.} The authors have been partially supported   by Conselleria d'Innovaci\'{o}, Universitats, Ci\`{e}ncia y Societat Digital, project AICO/2021/223.

\end{document}